\documentclass{amsart}

\usepackage{amssymb, amsmath}
\usepackage{mathrsfs}
\usepackage{amscd}
\usepackage{verbatim}
\usepackage{xcolor}
\usepackage{enumerate}

\allowdisplaybreaks
\theoremstyle{plain}
\newtheorem{theorem}{Theorem}[section]
\theoremstyle{remark}
\newtheorem{remark}[theorem]{Remark}
\newtheorem{example}[theorem]{Example}

\theoremstyle{plain}
\newtheorem{corollary}[theorem]{Corollary}
\newtheorem{lemma}[theorem]{Lemma}
\newtheorem{proposition}[theorem]{Proposition}
\newtheorem{definition}[theorem]{Definition}

\numberwithin{equation}{section}

\def \R{ \mathbb{R} }
\def \N{ \mathbb{N} }
\def \T{ \mathscr{T} }
\def \I{ \mathscr{I} }
\def \C{ \mathbb{C} }

\def \Ell{\rm Ell}

\renewcommand{\P}{{\mathbb P}}
\newcommand{\calK}{\mathcal{K}}

\def\avint_#1{\mathchoice%
      {\mathop{\kern 0.2em\vrule width 0.6em height 0.69678ex depth -0.58065ex
              \kern -0.8em \intop}\nolimits_{\kern -0.4em#1}}%
      {\mathop{\kern 0.1em\vrule width 0.5em height 0.69678ex depth -0.60387ex
              \kern -0.6em \intop}\nolimits_{#1}}%
      {\mathop{\kern 0.1em\vrule width 0.5em height 0.69678ex depth -0.60387ex
              \kern -0.6em \intop}\nolimits_{#1}}%
      {\mathop{\kern 0.1em\vrule width 0.5em height 0.69678ex depth -0.60387ex
              \kern -0.6em \intop}\nolimits_{#1}}}

\newcommand{\Rr}{\ensuremath{\mathcal{R}}}
\newcommand{\calL}{{\mathscr L}}

\newcommand{\one}{{{\bf 1}}}

\newcommand{\lb}{\langle}
\newcommand{\rb}{\rangle}

\newcommand{\calA}{\mathscr{A}}
\renewcommand{\MR}{\text{MR}^p((a,b),v)}
\newcommand{\MRT}{\text{MR}^p((0,T),v)}
\newcommand{\MRinfplus}{\text{MR}^p(\R_+,v)}
\newcommand{\MRinfR}{\text{MR}^p(\R,v)}

\DeclareMathOperator{\esssup}{ess.\,sup}

\begin{document}
\numberwithin{equation}{section}

\author{Chiara Gallarati}
\address{Delft Institute of Applied Mathematics\\
Delft University of Technology \\ P.O. Box 5031\\ 2600 GA Delft\\The
Netherlands} \email{C.Gallarati@tudelft.nl}

\author{Mark Veraar}
\email{M.C.Veraar@tudelft.nl}

\date\today

\title
[Maximal regularity]{Maximal regularity for non-autonomous equations with measurable dependence on time}

\begin{abstract}
In this paper we study maximal $L^p$-regularity for evolution equations with time-dependent operators $A$. We merely assume a measurable dependence on time. In the first part of the paper we present a new sufficient condition for the $L^p$-boundedness of a class of vector-valued singular integrals which does not rely on H\"ormander conditions in the time variable. This is then used to develop an abstract operator-theoretic approach to maximal regularity.

The results are applied to the case of $m$-th order elliptic operators $A$ with time and space-dependent coefficients. Here the highest order coefficients are assumed to be measurable in time and continuous in the space variables. This results in an $L^p(L^q)$-theory for such equations for $p,q\in (1, \infty)$.

In the final section we extend a well-posedness result for quasilinear equations to the time-dependent setting. Here we give an example of a nonlinear parabolic PDE to which the result can be applied.
\end{abstract}


%42B25  View Publications (1980-now) Maximal functions, Littlewood-Paley theory
%42B37  View Publications (2010-now) Harmonic analysis and PDE [See also 35-XX]
%46E30  View Publications (1973-now) Spaces of measurable functions (Lp-spaces, Orlicz spaces, Köthe function spaces, Lorentz spaces, rearrangement invariant spaces, ideal spaces, etc.)
%47A56  View Publications (1980-now) Functions whose values are linear operators (operator and matrix valued functions, etc., including analytic and meromorphic ones)
%47B55  View Publications (1973-1990) Operators on ordered spaces

%42B15  View Publications (1980-now) Multipliers
%47D06  View Publications (1991-now) One-parameter semigroups and linear evolution equations [See also 34G10, 34K30]
%42B37  View Publications (2010-now) Harmonic analysis and PDE [See also 35-XX]
%42B20  View Publications (1980-now) Singular and oscillatory integrals (Calderón-Zygmund, etc.)
%35B65  View Publications (1980-now) Smoothness and regularity of solutions
%35K30  View Publications (1973-now) Initial value problems for higher-order parabolic equations
%35K90  View Publications (2000-now) Abstract parabolic equations
%34G10  View Publications (1980-now) Linear equations [See also 47D06, 47D09]
%34G20  View Publications (1980-now) Nonlinear equations [See also 47Hxx, 47Jxx]
%35K55  View Publications (1973-now) Nonlinear parabolic equations


\thanks{The first author is supported by Vrije Competitie subsidy 613.001.206 and the second author by the Vidi subsidy 639.032.427 of the Netherlands Organisation for Scientific Research (NWO)}

\maketitle

%\tableofcontents

\section{Introduction}

In this paper we study maximal $L^p$-regularity of the Cauchy problem:
\begin{equation}\label{eq:introCauchygeneral}
\begin{aligned}
u'(t)+A(t)u(t)& =f(t),\ t\in (0,T)
\\ u(0)& =x.
\end{aligned}
\end{equation}
Here $(A(t))_{t\in (0,T)}$ is a family of closed operators on a Banach space $X_0$. We assume the operators have a constant domain $D(A(t)) = X_1$ for $t\in [0,T]$.

In recent years there has been much interest in maximal regularity techniques and
their application to nonlinear PDEs. Maximal regularity
can often be used to obtain a priori estimates which give global existence results. For
example, using maximal regularity it is possible to solve quasi-linear and fully nonlinear
PDEs by elegant linearization techniques combined with the contraction mapping
principle \cite{Am,Amoverview,CleLi,ClePru,Lun,Pruss02}.
This has found numerous applications in problems from mathematical
physics (e.g. fluid dynamics, reaction-diffusion equations, material science,
etc.\ see e.g.\ \cite{AbT,CleLi,DGHSS,GHHetc,HDR,LPS,Meynon,MSWM,Pierre,Pruss02,PrVVZ,Saal,YPZ}).
For maximal H\"older-regularity we refer the reader to \cite{AT2,Lun} and references therein. In this paper we focus on maximal $L^p$-regularity as this usually requires the least regularity of the data in PDEs.

An important step in the theory of maximal $L^p$-regularity was the discovery of an operator-theoretic characterization in terms of $\Rr$-boundedness properties of the differential operator $A$ due to Weis (see \cite{Weis01,Weis}). This characterization was proved for the class of Banach spaces with the UMD property. About the same time Kalton and Lancien discovered that not every sectorial operator $A$ on $X=L^q$ of angle $<\pi/2$ has maximal $L^p$-regularity (see \cite{KaLa,KaLa2} and \cite{Fack}), but their example is not a differential operator.

In the case $t\mapsto A(t)$ is (piecewise) continuous, one can study maximal $L^p$-regularity using perturbation arguments (see \cite{Am04,ACFP07,PS01}). In particular, in
\cite{PS01}, it was shown that maximal $L^p$-regularity of \eqref{eq:introCauchygeneral} is
equivalent to the maximal $L^p$-regularity for each operator $A(t_0)$ for $t_0\in [0,T]$ fixed. This, combined with the characterization of \cite{Weis} yields a very precise condition for maximal $L^p$-regularity.  The case where the domains $D(A(t))$ vary in time will not be considered in this paper. In that setting maximal $L^p$-regularity results can be obtained under certain H\"older regularity assumptions in the time variable (see \cite{PortalStr} and references therein).

In many real-life models, the differential operator $A$ has time-dependent coefficients, and
the dependence on time can be rather rough (e.g. the coefficient could be a stochastic process). If this is the case, the operator-theoretic
characterization of maximal regularity just mentioned does not apply or leads to unwanted restrictions.
In the present paper we develop a functional analytic approach to maximal $L^p$-regularity in the case $t\mapsto A(t)$ is only measurable (see Theorems \ref{thm:HSintro} and \ref{teostep1generalized} below). Our approach is based on the $L^p$-boundedness of a new class of vector-valued singular integrals of non-convolution type (see Theorem \ref{thm:singular}). It is important to note that we do not assume any H\"ormander conditions on the kernel in the time variable. For discussion and references on (vector-valued) singular integrals we refer the reader to Section \ref{sec:sing}.

When the time-dependence is just measurable, an operator-theoretic condition for maximal $L^p$-regularity is known only in the Hilbert space setting for $p=2$ (see \cite{Lions61, Lions68} and \cite[Section 5.5]{Ta1}). The assumption here is that $A$ arises from a coercive form $a(t, \cdot, \cdot):V\times V\to \C$ and $V\hookrightarrow X_0\hookrightarrow V'$. Unfortunately, this only yields a theory of maximal $L^2$-regularity on $V'$ in general (see \cite{fackler2016j} for a counterexample). In many situations one would like to have maximal $L^p$-regularity on $X_0$ and also for any $p\in (1, \infty)$. Results of this type have been obtained in \cite{ADLU, dier2014non, dier2016non, haak2014maximal} using regularity conditions on the form  in the time variable.

Most results will be presented in the setting of weighted $L^p$-spaces. For instance Theorems \ref{thm:2mintro} and  \ref{teomaxreghigherorder} we will present a weighted $L^p(L^q)$-maximal regularity result in the case $A$ is a $2m$-th order elliptic operator, assuming only measurability in the time variable and continuity in the space variable. Weighted results can be important for several reasons. Maximal $L^p$-regularity with a power weight $t^{\alpha}$ in time (e.g.\ see \cite{KPW, Meynon}) allows one to consider rather rough initial values. It can also be used to prove compactness properties which in turn can be used to obtain global existence of solutions. Another advantage of using weights comes from a harmonic analytic point of view. The theory of Rubio de Francia (see \cite{CMP} and references therein) enables one to extrapolate from weighted $L^p$-estimates for a single $p\in (1, \infty)$, to any $p\in (1, \infty)$. In Section \ref{sec:2mell} $A_p$-weights in space will be used to check $\Rr$-boundedness of certain integral operators. We refer to Theorem \ref{thm:weightedR} and Step 1 of the proof of Theorem \ref{teomaxreghigherorder} for details.
Weights in time will be used for extrapolation arguments more directly. For instance in step 4 of the proof of Theorem \ref{teomaxreghigherorder} and also the proof of Theorem \ref{thm:HSintro} at the end of Section \ref{subsec:tracesinitial}.

In the special case $X_0$ is a Hilbert space, our main result Theorem \ref{teostep1generalized} implies the following result.
\begin{theorem}\label{thm:HSintro}
Let $X_0$ be a Hilbert space. Assume $A:(0,\tau)\to \calL(X_1, X_0)$ is such that for all $x\in X_1$, $t\mapsto A(t) x$ is measurable and
\[c_1\|x\|_{X_1}\leq \|x\|_{X_0}+\|A(t)x\|_{X_0} \leq c_2\|x\|_{X_1}, \ \ \ t\in (0,\tau),  \ x\in X_1.\]
Assume there is an operator $A_0$ on $X_0$ with $D(A_0) = X_1$ which generates a contractive analytic semigroup $(e^{-zA_0})_{z\in \Sigma_{\theta}}$ which is such that $(A(t)-A_0)_{t\in (0,\tau)}$ generates an evolution system $(T(t,s))_{0\leq s\leq t\leq \tau}$ on $X_0$ which commutes with $(e^{-r A_0})_{r\geq 0}$.
\[e^{-rA_0}T(t,s) = T(t,s)e^{-rA_0}, \ \ 0\leq s\leq t\leq \tau, \ \ r\geq 0.\]
Then $A$ has maximal $L^p$-regularity for every $p\in (1, \infty)$, i.e.\ for every $f\in L^p(0,\tau;X_0)$ and $x\in (X_0, X_1)_{1-\frac1p,p}$ there exists a unique strong solution $u\in L^p(0,\tau;X_1) \cap W^{1,p}(0,\tau;X_0)\cap C([0,\tau];(X_0, X_1)_{1-\frac1p,p})$ of \eqref{eq:introCauchygeneral} and there is a constant $C$ independent of $f$ and $x$ such that
\begin{align*}
\|u\|_{L^p(0,\tau;X_1)} + \|u\|_{W^{1,p}(0,\tau;X_0)} + &\|u\|_{C([0,\tau];(X_0, X_1)_{1-\frac1p,p})}
\\ & \leq C\|f\|_{L^p(0,\tau;X_0)} + C\|x\|_{(X_0, X_1)_{1-\frac1p,p}}.
\end{align*}
\end{theorem}
The condition on $A(t) - A_0$ can be seen as an abstract ellipticity condition. The assumption that the operators are commuting for instance holds if $A(t)$ and $A_0$ are differential operators with coefficients independent of the space variable on $\R^d$. We will show that the space dependence can be put in later on by perturbation arguments.

In Section \ref{subsec:tracesinitial} we will derive this result from Theorem \ref{teostep1generalized} where the case of more general Banach spaces $X_0$ and weighted $L^p$-spaces is considered. Instead of assuming that $A_0$ generates an analytic contraction semigroup one could also assume that $A_0$ has a bounded $H^\infty$-calculus of angle $<\pi/2$.

\medskip

As an application of our main result we prove maximal $L^p$-regularity for the following class of parabolic PDEs:
\begin{equation}\label{eq:introCauchy}
\begin{aligned}
u'(t,x)+A(t)u(t,x)& =f(t,x),\ t\in (0,T), \ x\in \R^d,
\\ u(0,x)& =u_0(x),  \ \ x\in \R^d.
\end{aligned}
\end{equation}
Here
\begin{equation}\label{eq:operatorAintro}
A(t)u(t,x) = \sum_{|\alpha|\leq m} a_{\alpha}(t,x) D^{\alpha} u(t,x).
\end{equation}
For such concrete equations  with coefficients which depend on time in a measurable way, maximal $L^p$-regularity results can be derived using PDE techniques. Our results enable us to give an alternative approach to several of these problems. Moreover, we are the first to obtain a full $L^p(0,T;L^q(\R^d))$-theory, whereas previous papers usually only give results for $p=q$ or $q\leq p$ (see Remark \ref{rem:discussionHHKK} for discussion).

In the next result we will use condition (C) on $A$ which will be introduced in Section \ref{sec:2mell}.
It basically says that $A$ is uniformly elliptic and the highest order coefficients are continuous in space, but only measurable in time.
\begin{theorem}\label{thm:2mintro}
Let $T\in (0,\infty)$. Assume condition (C) on the family of operators  $(A(t))_{t\in (0,T)}$ given by \eqref{eq:operatorAintro}. Let $p,q\in (1, \infty)$. Then the operator $A$ has maximal $L^{p}$-regularity on $(0,T)$,
i.e.\ for every $f\in L^p(0,T;L^q(\R^d))$ and $u_0\in B^{s}_{q,p}(\R^d)$ with $s = m(1-\frac{1}{p})$, there exists a unique
\[u\in W^{1,p}(0,T;L^q(\R^d))\cap L^p(0,T;W^{m,q}(\R^d))\cap C([0,T];B^{s}_{q,p}(\R^d))\]
such that \eqref{eq:introCauchy} holds a.e.\ and there is a $C>0$ independent of $u_0$ and $f$ such that
\begin{equation}\label{eqmaxregintro}
\begin{aligned}
\|u\|_{L^p(0,T;W^{m,q}(\R^d))} & +\|u\|_{W^{1,p}(0,T;L^q(\R^d))} + \|u\|_{C([0,T];B^{s}_{q,p}(\R^d))}
\\ & \ \ \ \ \leq C \big(\|f\|_{L^{p}(\R;L^q(\R^d))} + \|u_0\|_{B^{s}_{q,p}(\R^d)}\big).
\end{aligned}
\end{equation}
\end{theorem}
The conditions on $f$ and $u_0$ are also necessary in the above result. Here $B^{s}_{q,p}(\R^d)$ denotes the usual Besov space (see \cite{Tr1} for details). The proof of Theorem \ref{thm:2mintro} is given at the end of Section \ref{sec:2mell}. It will be derived from Theorem \ref{teomaxreghigherorder} which is a maximal regularity result with weights in time and space. One can also consider systems instead of \eqref{eq:introCauchy}. The results in this case are more complicated and will be presented in \cite{GV2}.

\medskip

\textbf{Overview}
In Section \ref{secprelimin} we discuss preliminaries on weights, $\Rr$-boundedness and functional calculus. In Section \ref{sec:sing} we prove the $L^p$-boundedness of a new class of singular integrals. The main result on maximal $L^p$-regularity is presented in Section \ref{sec:maxLp}. In Section \ref{sec:2mell} we show how to use our new approach to derive maximal $L^p$-regularity for \eqref{eq:introCauchy}. Finally in Section \ref{sec:Quasi} we extend the result of \cite{CleLi} and \cite{Pruss02} on quasi-linear equations to the time-dependent setting.

\medskip

\textbf{Notation} Throughout this paper we will write $\calL(X,Y)$ for the space of all bounded linear operators mapping $X$ into $Y$. In the estimates below, $C$ can denote a constant which varies from line to line. We set $\N = \{1, 2, 3, \cdots\}$ and $\N_0 = \N\cup\{0\}$.

\medskip

\textbf{Acknowledgement} The authors thank Doyoon Kim for pointing out the references \cite{Kim10pot, Krylovheat}.
We also thank Nick Lindemulder and Jan Rozendaal for careful reading and useful comments. Finally, we thank the referee for his/her kind suggestions to improve the paper.

\section{Preliminaries}\label{secprelimin}

\subsection{$A_p$-weights}

Details on $A_p$-weights can be found in \cite[Chapter 9]{GrafakosModern} and \cite[Chapter V]{SteinHA}.

A {\em weight} is a locally integrable function on $\R^d$ with $w(x)\in (0,\infty)$ for a.e.\ $x \in \R^d$.
For a Banach space $X$ and $p\in [1, \infty]$, $L^p(\R^d,w;X)$ is the space of all strongly measurable functions $f:\R^d\to X$ such that
\begin{align*}
\|f\|_{L^p(\R^d,w;X)}=\Big(\int_{\R^d} \|f(x)\|^p w(x) \, dx\Big)^\frac{1}{p}<\infty \ \ \text{if $p\in [1, \infty)$},
\end{align*}
and $\displaystyle \|f\|_{L^\infty(\R^d,w;X)} = \esssup_{x\in \R^d} \|f(x)\|$.

For $p\in (1, \infty)$ a weight $w$ is said to be an {\em $A_{p}$-weight} if
\begin{align*}
   [w]_{A_{p}}=\sup_{Q} \avint_Q w(x) \, dx \Big(\avint_Q w(x)^{-\frac{1}{p-1}}\, dx \Big)^{p-1}<\infty.
\end{align*}
Here the supremum is taken over all cubes $Q\subseteq \R^d$ with axes parallel to the coordinate axes and $\avint_{Q} = \frac{1}{|Q|}\int_{Q}$. The extended real number $[w]_{A_{p}}$ is called the {\em $A_p$-constant}. The Hardy-Littlewood maximal operator is defined as
\begin{equation*}
  M(f)(x)=\sup_{Q \ni x}\avint_{Q}|f(y)|\, dy, \ \ \  f \in L^p(\R^d,w)
  \end{equation*}
with $Q\subseteq \R^d$ cubes as before. Recall that $w \in A_{p}$ if and only if the Hardy-Littlewood maximal operator $M$ is bounded on $L^p(\R^d,w)$.

The following simple extension of the extrapolation result from \cite[Theorem 3.9]{CMP} will be needed.
\begin{theorem}[Extrapolation]\label{thm:baseweight}
For every $\lambda\geq 0$, let $f_{\lambda},g_{\lambda}:\R^d \to \R_+$ be a pair of nonnegative, measurable functions and suppose that for some $p_0\in(1,\infty)$ there exist increasing functions $\alpha_{p_0}$, $\beta_{p_0}$ on $\R_+$ such that for all $w_0 \in A_{p_0}$ and all $\lambda\geq \beta_{p_0}([w_0]_{A_{p_0}})$,
\begin{equation}\label{eq:fgp0}
\|f_{\lambda}\|_{L^{p_0}(\R^d,w_0)} \leq \alpha_{p_0}([w_0]_{A_{p_0}}) \|g_{\lambda}\|_{L^{p_0}(\R^d,w_0)}.
\end{equation}
Then for all $p \in (1,\infty)$ there is a constant $c_{p,d}\geq 1$ such that for all $w \in A_{p}$, and all $\lambda\geq \beta_{p_0}(\phi([w]_{A_p}))$
\begin{equation*}
\|f_{\lambda}\|_{L^p(\R^d,w)}\leq 4 \alpha_{p_0}(\phi([w]_{A_p})) \|g_{\lambda}\|_{L^p(\R^d,w)},
\end{equation*}
where $\phi(x) = c_{p,d} x^{\frac{p_0-1}{p-1}+1}$.
\end{theorem}
Note that \cite[Theorem 3.9]{CMP} corresponds to the case that $f_{\lambda}$ and $g_{\lambda}$ are constant in $\lambda$. To obtain the above extension one can check that in the proof \cite[Theorem 3.9]{CMP} for given $p$ and $w\in A_p$, the $A_{p_0}$-weight $w_0$ which is constructed satisfies $[w_0]_{A_{p_0}} \leq \phi([w]_{A_p})$. This clarifies the restriction on the $\lambda$'s.

Below estimates of the form \eqref{eq:fgp0} with increasing function $\alpha_{p_0}$ will appear frequently. In this situation we say there is an {\em $A_{p_0}$-consistent} constant $C$ such that
\[\|f\|_{L^{p_0}(\R^d,w_0)} \leq C\|g\|_{L^{p_0}(\R^d,w_0)}.\]
Note that the $L^p$-estimate obtained in Theorem \ref{thm:baseweight} is again $A_p$-consistent for all $p\in (1, \infty)$.

The following simple observation will be applied frequently. For a bounded Borel set $A\subset\R^d$ and for every $f\in L^p(\R^d,w;X)$ one has $\one_A f\in L^1(\R^d;X)$ and by  H\"older's inequality
\[\|\one_A f\|_{L^1(\R^d;X)}\leq C_{w,A} \|f\|_{L^p(\R^d,w;X)}.\]

A linear subspace $Y\subseteq X^*$ is said to be norming for $X$ if for all $x\in X$, $\|x\| = \sup\{|\lb x, x^*\rb|: x^*\in Y, \|x^*\|\leq 1\}$. The following simple duality lemma will be needed.
\begin{lemma}\label{lem:norming}
Let $p,p'\in [1, \infty]$ be such that $\frac{1}{p} + \frac{1}{p'} = 1$. Let $v$ be a weight and let $v' = v^{-\frac{1}{p-1}}$. Let $Y\subseteq X^*$ be a subspace which is norming for $X$. Then setting
\[\lb f, g\rb = \int_{\R} \lb f(t), g(t)\rb \, dt,  \ \  f\in L^p(\R,v;X), \ \  g\in L^{p'}(\R,v';X^*),\]
the space $L^{p'}(\R,v';X^*)$ can be isometrically identified with a closed subspace of $L^p(\R,v;X)^*$.
Moreover, $L^{p'}(\R,v';Y)$ is norming for $L^p(\R,v;X)$.
\end{lemma}

\subsection{$\Rr$-boundedness and integral operators\label{sec:Rbdd}}

In this section we recall the definition of $\Rr$-boundedness (see \cite{CPSW, DHP, KW} for details).

A sequence of independent random variables $(r_{n})_{n\geq 1}$ on a probability space $(\Omega, \mathscr{A}, \mathbb{P})$ is called a {\em Rademacher sequence} if $\P(r_n = 1) = \P(r_n  = -1) = \frac12$.

Let $X$ and $Y$ be Banach spaces. A family of operators $\T \subseteq \calL(X,Y)$ is said to be {\em $\Rr$-bounded} if there exists a constant $C$ such that for all $N\in \N$, all sequences $(T_n)_{n=1}^ N$ in $\T$ and $(x_n)_{n=1}^N$ in $X$,
\begin{align}\label{eq:Rbdd}
\Big\|\sum_{n=1}^N r_n T_n x_n \Big\|_{L^2(\Omega;Y)}\leq C\Big\|\sum_{n=1}^N r_n x_n \Big\|_{L^2(\Omega;X)}
\end{align}
The least possible constant $C$ is called the {\em $\Rr$-bound} of $\T$ and is
denoted by $\Rr(\T)$. Recall the Kahane-Khintchine inequalities (see \cite[11.1]{DJT}): for every $p,q\in (0, \infty)$, there exists a $\kappa_{p,q}>0$ such that
\begin{align}\label{eq:KahKin}
\Big\|\sum_{n=1}^N r_n x_n\Big\|_{L^p(\Omega;X)} \leq \kappa_{p,q}\Big\|\sum_{n=1}^N r_n x_n\Big\|_{L^q(\Omega;X)}.
\end{align}
Therefore, the $L^2(\Omega;X)$-norms in \eqref{eq:Rbdd} can be replaced by $L^p(\Omega;X)$, to obtain an equivalent definition up to a constant depending on $p$.

Every $\Rr$-bounded family of operators is uniformly bounded. A converse holds for Hilbert spaces $X$ and $Y$: every uniform bounded family of operators is automatically $\Rr$-bounded.

The $\Rr$-boundedness of a certain family of integral operators plays a crucial role in this paper. Let $\calK$ be the class of kernels $k\in L^1(\R)$ for which $|k|*f\leq Mf$ for all simple functions $f:\R\to \R_+$, where $M$
denotes the Hardy-Littlewood maximal operator. The next example gives an important class of kernels which are in $\calK$.

\begin{example}\label{ex:kernel}
Let $k:(0,\infty)\times \R\to \C$, be such that $|k(u,t)|\leq h(\tfrac{|t|}{u})\frac{1}{u},\ u>0$, where $h\in L^{1}(\R_{+})\cap C_b(\R_+)$, $h$ has a maximum in $x_{0}\in [0,\infty)$ and $h$ is radially decreasing on $[x_{0},\infty)$. Then,
\begin{align*}
\int_{0}^{\infty}\sup_{|t|\geq x}|k(u,t)| \, dx&\leq \int_{0}^{\infty}\sup_{t\geq x}|h(\tfrac{t}{u})| \, \frac{ dx}{u}
=\int_{0}^{\infty}\sup_{s\geq \frac{x}{u}}|h(s)|\, \frac{dx}{u}
=\int_{0}^{\infty}\sup_{s\geq y}|h(s)|\, dy
\\ &=\int_{0}^{x_{0}}\sup_{s\geq y}|h(s)| \, dy+\int_{x_{0}}^{\infty} |h(y)| \, dy
 =x_{0} |h(x_{0})|+\|h\|_{L^{1}(x_0,\infty)}.
\end{align*}
Now by \cite[Proposition 4.5]{NVWR} we find $\{\frac{k(u,\cdot)}{C}:u>0\}\subseteq \calK$ with $C = x_{0} |h(x_{0})|+\|h\|_{L^{1}(x_0,\infty)}$.
\end{example}

Suppose $T:\{(t,s)\in \R^2: t\neq s\}\to \calL(X)$ is such that for all $x\in X$, $(t,s)\mapsto T(t,s)x$ is measurable. For $k\in \calK$ let
\begin{equation}\label{eq:IkTdefprelim}
I_{k T} f(t) = \int_{\R} k(t-s) T(t,s) f(s)\, ds.
\end{equation}
Consider the family of integral operators $\I:=\{I_{k T}: k\in \calK\}\subseteq\calL(L^{p}(\R;X))$. The $\Rr$-boundedness of such families $\I$ of operators will play an important role in Section \ref{sec:sing}.

\begin{proposition}\label{prop:uniformTts}
If $\{T(t,s):s,t\in\R\}$ is uniformly bounded on $X$, then $\I$ is uniformly bounded on $L^p(\R,v;X)$ for every $p\in (1, \infty)$ and $v\in A_p$. Moreover, it is also uniform bounded on $L^1(\R;X)$.
\end{proposition}
\begin{proof}
For any $p\in (1, \infty)$, note that
\begin{align*}
\|I_{kT} f(t)\|_{X} & \leq \int_{\R} |k(t-s)| \|T(t,s)  f(s)\|_{X} \, ds
\\ & \leq C \int_{\R} |k(t-s)| \|f(s)\|_{X} \, ds \leq C M (\|f\|_{X})(t).
\end{align*}
for a.e.\ $t\in \R$. Therefore the uniform boundedness of $I_{kT}$ follows from the boundedness of the maximal operator. The case $v\equiv 1$ and $p=1$ follows from Fubini's theorem and the fact that $\|k\|_{L^1(\R)}\leq 1$ (see \cite[Lemma 4.3]{NVWR}).
\end{proof}

The $\Rr$-boundedness of \eqref{eq:IkTdefprelim} has the following simple extrapolation property:

\begin{proposition}\label{prop:pindRbdd}
Let $p_0\in (1, \infty)$. If for all $v\in A_{p_0}$, $\I\subseteq \calL(L^{p_0}(\R,v;X))$ is $\Rr$-bounded by a constant which is $A_{p_0}$-consistent, then for every $p\in (1, \infty)$ and $v\in A_p$, $\I\subseteq \calL(L^{p}(\R,v;X))$ is $\Rr$-bounded by a constant which is $A_p$-consistent.
\end{proposition}
\begin{proof}
The special structure of $\I$ will not be used in this proof. Let $I_1, \ldots, I_N\in \I$, $f_1, \ldots, f_N\in L^p(\R,v;X)$ and let
\[F_p(t) = \Big\|\sum_{n=1}^N r_n I_n f_n(t)\Big\|_{L^p(\Omega;X)} \  \ \text{and} \ \ G_p(t) = \Big\|\sum_{n=1}^N r_n f_n(t)\Big\|_{L^p(\Omega;X)}.\]
Then the assumption combined with Fubini's theorem yields that for all $v\in A_{p_0}$,
\[\|F_{p_0}\|_{L^{p_0}(\R,v)} \leq C \|G_{p_0}\|_{L^{p_0}(\R,v)},\]
where $C$ is a constant which is $A_{p_0}$-consistent. Therefore, by Theorem \ref{thm:baseweight} we find that for each $p\in(1, \infty)$, there is an $A_p$-consistent constant $C'$ (depending only on $C$) such that
\begin{equation}\label{eq:Fp0Gp0}
\|F_{p_0}\|_{L^{p}(\R,v)} \leq C' \|G_{p_0}\|_{L^{p}(\R,v)}.
\end{equation}
Now by \eqref{eq:KahKin}, $F_{p}\leq \kappa_{p, p_0} F_{p_0}$, $G_{p_0}\leq \kappa_{p_0, p}G_{p}$, and the result follows from \eqref{eq:Fp0Gp0} and another application of Fubini's theorem.
\end{proof}

In \cite{GLV} the following simple sufficient condition for $\Rr$-boundedness of such families was obtained in the case $X = L^q$.
\begin{theorem}\label{thm:weightedR}
Let $\mathcal{O}\subseteq \R^d$ be open. Let $q_0\in (1, \infty)$ and let $\{T(t,s):s,t\in \R\}$ be a family of bounded operators on $L^{q_0}(\mathcal{O})$. Assume that for all $A_{q_0}$-weights $w$,
\begin{equation}\label{eq:weightedcond}
\|T(t,s)\|_{\calL(L^{q_0}(\mathcal{O},w))}\leq C,
\end{equation}
where $C$ is $A_{q_0}$-consistent and independent of $t,s\in \R$. Then the family of integral operators
$\I = \{I_{k T}: k\in \calK\}\subseteq \calL(L^p(\R,v;L^q(\mathcal{O},w)))$ as defined in \eqref{eq:IkTdefprelim} is $\Rr$-bounded for all $p,q\in (1, \infty)$ and all $v\in A_p$ and $w\in A_q$.
Moreover, in this case the $\Rr$-bounds $\Rr(\I)$ are $A_{p}$- and $A_q$-consistent.
\end{theorem}
The proof of this result is based on extrapolation techniques of Rubio de Francia.
As for fixed $t,s\in \R$, $T(t,s)$ on $L^q(\mathcal{O})$ is usually defined by a singular integral of convolution type in $\R^d$, one can often apply Calder\'on-Zygmund theory and multiplier theory to verify \eqref{eq:weightedcond}. In this case it is usually not more difficult to prove the boundedness for all $A_q$-weights, than just $w=1$. The reason for this is that for large classes of operators, boundedness implies weighted boundedness
(see \cite[Theorem IV.3.9]{GarciaRubio}, \cite[Theorem 9.4.6]{GrafakosModern} and \cite[Corollary 2.10]{HaHy}). Another situation where weights are used to obtain $\Rr$-boundedness can be found in \cite{HHH,Frohlich}.

\begin{example}\label{ex:boundedT}
For a bounded measurable function $\theta:\R^2\to \C$ let $T(t,s) f = \theta(t,s) f$, $f\in  L^{q_0}(\R^d,w)$. Then \eqref{eq:weightedcond} holds and hence Theorem \ref{thm:weightedR} implies that $\I\subseteq \calL(L^p(\R,v;L^q(\R^d,w)))$ is $\Rr$-bounded for all $p,q\in (1, \infty)$ and all $v\in A_p$ and $w\in A_q$.
\end{example}

\subsection{Sectorial operators and $H^\infty$-calculus\label{FuncCalc}}

Let $X$ be a Banach space. We briefly recall the definition of the $H^\infty$-calculus which was developed by McIntosh and collaborators (see e.g.\ \cite{ADM,AMcN,CDMcY96,McI}). We refer to \cite{Haase:2, KW} for an extensive treatment of the subject.
For $\theta\in (0,\pi)$ we set
\[\Sigma_\theta = \{z\in \C\setminus\{0\}: \ |\arg(z)| < \theta\},\]
where $\arg:\C\setminus\{0\}\to (-\pi,\pi]$. A closed densely defined linear operator
$(A, D(A))$ on $X$ is said to be {\em sectorial of type $\sigma\in (0,\pi)$} if it is injective and has dense range,
its spectrum is contained in $\overline{\Sigma_\sigma}$, and for all $\sigma'\in(\sigma,\pi)$ the set
$$ \big\{z(z+A)^{-1}: \ z\in \C\setminus\{0\}, \  |\arg(z)|> \sigma'\big\}$$
is uniformly bounded by some constant $C_A$. The infimum of all $\sigma\in (0,\pi)$ such that $A$ is sectorial of type
$\sigma$ is called the {\em sectoriality angle} of $A$. If $\sigma<\pi/2$, then by \cite[Proposition 2.1.1]{Lun}, $A$ generates an analytic strongly continuous semigroup $T(z) = e^{-zA}$ for $\arg(z)<\pi/2 - \sigma$ and
\begin{align}\label{eq:semigroupbound}
\|T(t)\| \leq C_A C_{\sigma},  \ \ t\geq 0.
\end{align}

Let $H^\infty(\Sigma_\theta)$ denote the Banach space of all bounded analytic functions
$f:\Sigma_\theta\to \C$, endowed with the supremum norm. Let
$H_0^\infty(\Sigma_\theta)$ denote the linear subspace of all
$f\in H^\infty(\Sigma_\theta)$ for which there exists $\varepsilon>0$ and $C\ge 0$ such that
\[|f(z)| \le \frac{C|z|^\varepsilon}{(1+|z|)^{2\varepsilon}}, \quad z\in \Sigma_\theta.\]
If $A$ is sectorial of type $\sigma_0\in (0,\pi)$, then for all $\sigma\in (\sigma_0,\pi)$ and
$f\in H_0^\infty(\Sigma_{\sigma})$ we define the bounded operator $f(A)$ by
\[ f(A) = \frac1{2\pi i}\int_{\partial \Sigma_{\sigma}} f(z) (z+A)^{-1}\,dz.\]

A sectorial operator $A$ of type $\sigma_0\in (0,\pi)$
is said to have a {\em  bounded $H^\infty(\Sigma_{\sigma})$-calculus}) for $\sigma\in (\sigma_0,\pi)$ if there exists a $C\ge 0$ such that
\[ \| f(A)\| \le C\|f\|_{H^\infty(\Sigma_\sigma)}, \ \ \ f\in H_0^\infty(\Sigma_\sigma).\]
If $A$ has a bounded $H^\infty(\Sigma_{\sigma})$-calculus, then the mapping $f\mapsto f(A)$
extends to a bounded algebra homomorphism from
$H^\infty(\Sigma_\sigma)$ to $\calL(X)$ of norm $\leq C$.

Many differential operators on $L^q$-spaces with $q\in (1, \infty)$ are known to have a bounded $H^\infty$-calculus (see \cite{DHP,KW} and the survey \cite{HinfWeis}).
The case $A = -\Delta$ on $L^p(\R^d,w)$ has a bounded $H^\infty$-calculus of arbitrary small angle $\sigma\in (0,\pi)$ for every $w\in A_p$ and $p\in (1, \infty)$. This easily follows from the weighted version of Mihlin's multiplier theorem (see \cite[Example 10.2]{KW} and \cite[Theorem IV.3.9]{GarciaRubio}).
For instance, it includes all sectorial operators $A$ of angle $<\pi/2$ for which $e^{-tA}$ is a positive contraction (see \cite{KWcalc}).

\section{A class of singular integrals with operator-valued kernel\label{sec:sing}}
Let $X$ be a Banach space. In this section we will study a class of singular integrals of the form
\begin{equation}\label{eq:singular}
I_K f(t) =  \int_{\R} K(t,s) f(s) \, ds, \ \ \ t\in \R,
\end{equation}
where $K:\{(t,s): t\neq s\}\to \calL(X)$ is an operator-valued kernel. If a kernel $L$ depends on one variable we write $I_L = I_K$ where $K(t,s) = L(t-s)$.

There is a natural generalization of the theory of singular integrals of {\em convolution type} to the vector-valued setting (see \cite{HyWe07}). In the case the singular integral is of {\em non-convolution type}, the situation is much more complicated. An extensive treatment can be found in \cite{Hyt:Tb, Hytonen:nonhomog, HytWeis:T1}, where $T1$-theorems \cite{DavJou84} and $Tb$-theorems \cite{DJS85} have been obtained in an infinite dimensional setting. Checking the conditions of these theorems can be hard. For instance, from \cite{TaoMei06} it follows that the typical BMO conditions one needs to check, have a different behavior in infinite dimensions. Our motivation comes from the application to maximal $L^p$-regularity of \eqref{eq:introCauchy}. At the moment we do not know whether the $T1$-theorem and $Tb$-theorem can be applied to study maximal $L^p$-regularity for the time dependent problems we consider. Below we study a special class of singular integrals with operator-valued kernel for which we prove $L^p$-boundedness. The assumptions on $K$ are formulated in such a way that they are suitable for proving maximal $L^p$-regularity of \eqref{eq:introCauchy} later on.

\subsection{Assumptions}

The assumptions in the main result of this section are as follows.

\let\ALTERWERTA\theenumi
\let\ALTERWERTB\labelenumi
\def\theenumi{(H1)}
\def\labelenumi{(H1)}
\begin{enumerate}
\item\label{as:HXpv} Let $X$ be a Banach space and let $p\in [1, \infty)$ and\footnote{For the case $p=1$, the convention will be that $v\equiv 1$.} $v\in A_p$.
\end{enumerate}
\let\theenumi\ALTERWERTA
\let\labelenumi\ALTERWERTB

\let\ALTERWERTA\theenumi
\let\ALTERWERTB\labelenumi
\def\theenumi{(H2)}
\def\labelenumi{(H2)}
\begin{enumerate}
\item\label{as:Kfact} The kernel $K$ factorizes as
\begin{equation}\label{eq:factorizeKernel}
K(t,s) = \frac{\phi_0(|t-s|A_0) T(t,s) \phi_1(|t-s|A_1)}{t-s},  \ \ \ (t,s)\in\R^2, t\neq s.
\end{equation}
Here $A_0$ and $A_1$ are sectorial operators on $X$ of angle $<\sigma_0$ and $<\sigma_1$ respectively, and $\phi_j\in H^\infty_0(\Sigma_{\sigma_j'})$ and $\sigma_j'\in (\sigma_j, \pi)$ for $j=0, 1$. Moreover, we assume $T:\{(t,s):t\neq s\}\to \calL(X)$ is uniformly bounded and for all $x\in X$, $\{(t,s):t\neq s\} \mapsto T(t,s)x$ is strongly measurable.
\end{enumerate}
\let\theenumi\ALTERWERTA
\let\labelenumi\ALTERWERTB

\let\ALTERWERTA\theenumi
\let\ALTERWERTB\labelenumi
\def\theenumi{(H3)}
\def\labelenumi{(H3)}
\begin{enumerate}
\item\label{as:Hinfty} Assume $X$ has finite cotype. Assume $A_j$ has a bounded $H^\infty(\Sigma_{\sigma_j})$-calculus with $\sigma_j\in [0,\pi)$ for $j=0,1$.
\end{enumerate}
\let\theenumi\ALTERWERTA
\let\labelenumi\ALTERWERTB

\let\ALTERWERTA\theenumi
\let\ALTERWERTB\labelenumi
\def\theenumi{(H4)}
\def\labelenumi{(H4)}
\begin{enumerate}
\item\label{as:Rbdd}
Assume the family of integral operators
$\I:=\{I_{k T}: k\in \calK\}\subseteq\calL(L^{p}(\R,v;X))$ is $\Rr$-bounded.
\end{enumerate}
\let\theenumi\ALTERWERTA
\let\labelenumi\ALTERWERTB

The class of kernels $\calK$ is as defined in Section \ref{sec:Rbdd}. Recall from \eqref{eq:singular} that
\begin{equation}\label{eq:IkTdef}
I_{k T} f(t) = \int_{\R} k(t-s) T(t,s) f(s)\, ds.
\end{equation}
Since $T$ is uniformly bounded, the operator $I_{kT}$ is bounded on $L^p(\R,v;X)$ by Proposition \ref{prop:uniformTts}.

\begin{remark}\
\begin{enumerate}
\item The class of Banach spaces with finite cotype is rather large. It contains all $L^p$-spaces, Sobolev, Besov and Hardy spaces as long as the integrability exponents are in the range $[1, \infty)$. The spaces $c_0$ and $L^\infty$ do not have finite cotype. The cotype of $X$ will be applied in order to have estimates for certain continuous square functions (see \eqref{eq:squarefH}). Details on type and cotype can be found in \cite{DJT}.
\item In the theory of singular integrals in a vector-valued setting one usually assumes $X$ is a UMD space. Note that every UMD has finite cotype and nontrivial type by the Maurey-Pisier theorem (see \cite{DJT}).
\item A sufficient condition for the $R$-boundedness condition in the case $X = L^q$ can be deduced from Theorem \ref{thm:weightedR}.
\item In \ref{as:Kfact}, $\phi_j(|t-s|A_j)$ could be replaced by $\phi_j((t-s)A_j)$ if the $A_j$'s are bisectorial operators. On the other hand, one can also consider $T(t,s) \one_{\{s<t\}}$ and $T(t,s) \one_{\{t<s\}}$ separately. Indeed, the hypothesis \ref{as:HXpv}--\ref{as:Rbdd} holds for these operators as well whenever they hold for $T(t,s)$.
\end{enumerate}
\end{remark}

\begin{example}
Typical examples of functions $\phi_j$ which one can take are $\phi_j(z) = z^{\alpha} e^{-z}$ for $j=0, 1$. If $T(t,s) = I \one_{\{s<t\}}$, then for $A = A_0 = A_1$ one would have
\[K(t,s) = (t-s)^{2\alpha-1} A^{2\alpha} e^{-2(t-s)A} \one_{\{s<t\}}.\]
This kernel satisfies $\|K(t,s)\|\sim (t-s)^{-1}$ for $t$ close to $s$. If one takes $T(t,s)$ varying in $t$ and $s$ one might view it as a multiplicative perturbation of the above kernel.
\end{example}

The following simple observation shows that $I_K$ as given in \eqref{eq:singular} can be defined on $L^p(\R;D(A_1)\cap R(A_1))$, where $D(A_1)$ denotes the domain of $A_1$ and $R(A_1)$ the range of $A_1$.
\begin{lemma}\label{lem:defondense}
Under the assumptions \ref{as:HXpv} and \ref{as:Kfact}, $I_K$ is bounded as an operator from $L^p(\R,v;D(A_1)\cap R(A_1))$ into $L^p(\R,v;X)$.
\end{lemma}
\begin{proof}
As $\phi_1\in H^\infty_0(\Sigma_{\sigma_1'})$ we can find a constant $C$ and $\varepsilon\in (0,1)$ such that $|\phi_1(z)|\leq C |z|^{\varepsilon} |1+z|^{-2\varepsilon}$.
One can check that for all $x\in D(A_1)\cap R(A_1)$,
\begin{equation}\label{eq:estpsi}
\|\phi_1(tA_1) x\|\leq C \min\{t^{\varepsilon}, t^{-\varepsilon}\} (\|x\|+ \|A_1x\| + \|A_1^{-1}x\|), \  \ t>0.
\end{equation}
Now since  $\phi_0\in H^\infty_0(\Sigma_{\sigma_0'})$ and $\|T(t,s)\|$ is uniformly bounded we obtain
\begin{align*}
  |t-s|\, \|K(t,s)x\| &\leq \|\phi_0(|t-s|A_0)\| \, \|T(t,s)\| \, \|\phi_1(|t-s|A_1) x\|
  \\ & \leq C \min\{|t-s|^{\varepsilon}, |t-s|^{-\varepsilon}\} (\|x\|+ \|A_1x\| + \|A_1^{-1}x\|).
\end{align*}
Therefore, $K:\{(t,s):t\neq s\}\to \calL(R(A_1)\cap D(A_1), X)$ is essentially nonsingular, and the assertion of the lemma easily follows from \cite[Theorem 2.1.10]{GrafakosClassical} and the boundedness of the Hardy-Littlewood maximal operator for $p\in (1, \infty)$. The case $p=1$ follows from Young's inequality.
\end{proof}

\subsection{Main result on singular integrals}

\begin{theorem}\label{thm:singular}
Assume \ref{as:HXpv}-\ref{as:Rbdd}. Then $I_K$ defined by \eqref{eq:singular} extends to a bounded operator on $L^p(\R,v;X)$.
\end{theorem}
The proof is inspired by the recent solution to the stochastic maximal $L^p$-regularity problem given in \cite{JMW}.

Before we turn to the proof, we have some preliminary results and remarks.
\begin{example}
Assume \ref{as:Kfact} and \ref{as:Hinfty}. If $T(t,s)$ is as in Example \ref{ex:boundedT} then \ref{as:Rbdd} holds. Therefore, $I_K$ is bounded by Theorem \ref{thm:singular}.  Surprisingly, we do not need any smoothness of the mapping $(t,s)\mapsto K(t,s)$ in this result. In particular we do not need any regularity conditions for $K(t,s)$ (such as H\"ormander's condition) in $(t,s)$.
\end{example}

Recall the following Poisson representation formula (see \cite[Lemma 4.1]{JMW}).
\begin{lemma}\label{lemma4.1gen}
Let $\alpha\in(0,\pi)$ and $\alpha'\in(\alpha,\pi]$ be given, let $E$ be a Banach space and let $f:\Sigma_{\alpha'}\rightarrow E$ be a bounded analytic function. Then, for all $s>0$ we have
\begin{equation}
f(s)=\sum_{j\in\{-1,1\}}\frac{j}{2}\int_{0}^{\infty}k_{\alpha}(u,s)f(ue^{ij\alpha})du,
\nonumber
\end{equation}
where $k_{\alpha}:\R_{+}\times\R_{+}\rightarrow\R$ is given by
\begin{equation}\label{eq4.1gen}
k_{\alpha}(u,t)=\frac{(t/u)^{\frac{\pi}{2\alpha}}}{(t/u)^{\frac{\pi}{\alpha}}+1}\frac{1}{\alpha u}.
\end{equation}
\end{lemma}

\begin{remark}
In the special case $X = L^q(S)$ with $q\in (1, \infty)$, we present some identification of spaces which can be used to simplify the proof below. This might be of use to readers who are only interested in $L^q$-spaces.
First of all one can use the usual adjoint $^*$ instead of the moon adjoint $\#$ in the proof below.
In this case one can take
\begin{align*}
\gamma(\tfrac{du}{u};X) &= L^q(S;L^2(0,\infty;\tfrac{du}{u})),
\\ \gamma(\tfrac{du}{u};X)^* &= L^{q'}(S;L^2(0,\infty;\tfrac{du}{u})),
\\ \gamma(\tfrac{du}{u};L^p(\R,v;X)) &= L^p(\R,v;L^q(S;L^2(0,\infty;\tfrac{du}{u}))).
\end{align*}
The $\gamma$-multiplier theorem which is applied below in \eqref{eq:gammamultiplier} can be replaced by \cite[4a]{Weis01} in this case. Finally, the estimates in \eqref{eq:squarefH} can be found in \cite{LeM04} in this special case.
\end{remark}

\begin{proof}[Proof of Theorem \ref{thm:singular}]
{\em Step 1}: By density it suffices to prove $\|I_K f\|_{L^p(\R,v;X)}\leq C \|f\|_{L^p(\R,v;X)}$ with $C$ independent of $f\in L^p(\R,v;D(A_1)\cap R(A_1))$. Note that by Lemma \ref{lem:defondense}, $I_K$ is well defined on this subspace.

\medskip

{\em Step 2}: Fix $0<\alpha<\alpha'\leq\min\{\sigma_{0}'-\sigma_{0},\sigma_{1}'-\sigma_{1}\}$.
First, since $z\rightarrow \phi_0(zA_0) T(t,s) \phi_1(zA_1)$ is analytic and bounded on $\Sigma_{\alpha'}$, by Lemma \ref{lemma4.1gen}, for $x\in D(A_1)\cap R(A_1)$ and $z>0$,
\begin{align*}
\phi_0(z A_0) & T(t,s) \phi_1(zA_1)x =\sum_{j\in\{-1,1\}}\frac{j}{2}\int_{0}^{\infty} \Phi_{0,j}(u)k_{\alpha}(u,z) T(t,s) \Phi_{1,j}(u) x\, du
\end{align*}
with $k_{\alpha}(u,t)$ as in \eqref{eq4.1gen} and $\Phi_{k,j}(u) = \phi_{k}(ue^{i j \alpha}A_k)$ for $j\in \{-1, 1\}$ and $k\in \{0,1\}$. Together with \ref{as:Kfact} this yields the following representation of $K(t,s)x$  for $x\in D(A_1)\cap R(A_1)$:
\[K(t,s)x = \sum_{j\in\{-1,1\}}\frac{j}{2}\int_{0}^{\infty} \Phi_{0,j}(u) S_u(t,s) \Phi_{1,j}(u)x\, \frac{du}{u},\]
where $S_u(t,s):= \tilde{k}_{\alpha}(u,t-s) T(t,s)$ with $\tilde{k}_{\alpha}(u,t):= k_{\alpha}(u,|t|) \tfrac{u}{t}$ and $k_{\alpha}$ is defined as in \eqref{eq4.1gen}. Moreover, the kernels $\tilde{k}_{\alpha}(u,\cdot)$
satisfy
\[|\tilde{k}_{\alpha}(u,t)| \leq \alpha^{-1} h_{\alpha}(\tfrac{t}{u})u^{-1},  \ \ \ u,t>0,\]
where $h_{\alpha}(x)=\frac{x^{\beta-1}}{x^{2\beta}+1}$ and $\beta:=\frac{\pi}{2\alpha}>0$.
Extending $k_{\alpha}(u,t)$ as zero for $t<0$, by Example \ref{ex:kernel} we find that $\tilde{k}_{\alpha}(u,\cdot)\in  \calK$. Indeed, substituting $y = x^{\beta}$, we obtain
\[\|h\|_{L^1(0,\infty)} = \int_0^\infty \frac{x^{\beta-1}}{x^{2\beta}+1} \, dx = \frac{1}{\beta}\int_0^\infty \frac{1}{y^{2}+1} \, dy = \alpha.\]
Therefore, the following representation holds for the singular integral
\begin{align*}
I_K f = \sum_{j\in\{-1,1\}}\frac{j}{2}\int_{0}^{\infty} \Phi_{0,j}(u) I_{S_u} [\Phi_{1,j}(u) f]\, \frac{du}{u},
\end{align*}
where $f\in L^p(\R,v;D(A_1)\cap R(A_1))$.

\medskip

{\em  Step 3}:
Let $Y_1 = L^p(\R,v;X)$ and $Y_2 = L^{p'}(\R,v';X^\#)$, where $X^\# = \overline{D(A^*_0)} \cap \overline{R(A^*_0)}$ is the moondual of $X$ with respect to $A_0$ (see \cite[Appendix A]{KW}) and $v' = v^{-\frac{1}{p-1}}$. For $g\in Y_2$ write $\lb f, g\rb_{Y_1, Y_2} = \int_{\R} \lb f(t), g(t)\rb \, dt$. In this way $Y_2$ can be identified with an isometric closed subspace of $Y_1^*$. Note that by \cite[Proposition 15.4]{KW}, $X^\#$ is norming for $X$ and hence Lemma \ref{lem:norming} implies that $Y_2$ is norming for $Y_1$.
For fixed $g\in Y_2$ it follows from Fubini's theorem and $\gamma$-duality (see \cite[Sections 2.3 and 2.6]{HaHa} and \cite[Section 5]{KalW04}),
\begin{align*}
|\lb  I_K f, g\rb_{Y_1,Y_2}|& \leq  \sum_{j\in\{-1,1\}}\frac{1}{2} \Big|\int_{\R}\int_{0}^{\infty} \lb \Phi_{0,j}(u) I_{S_u}[\Phi_{1,j}(u) f](t), g(t)\rb \, \frac{du}{u}   \, dt\Big|
\\ &= \sum_{j\in\{-1,1\}}\frac{1}{2} \Big|\int_{0}^{\infty} \lb I_{S_u}[\Phi_{1,j}(u) f], \Phi_{0,j}(u)^\# g\rb \, \frac{du}{u}   \Big|
\\ & = \sum_{j\in\{-1,1\}}\frac{1}{2} \|I_{S_u}[\Phi_{1,j}(u) f]\|_{\gamma(\R_+,\frac{du}{u};Y_1)} \|\Phi_{0,j}(u)^\# g\|_{\gamma(\R_+,\frac{du}{u};Y_1)^*}.
\end{align*}
Here $\Phi_{0,j}(u)^\#:=\phi_0(ue^{ij\alpha}A^\#_0)$.
By \ref{as:Rbdd} the family $\{I_{S_u}:u>0\}$ is $\Rr$-bounded by some constant $C_T$. Therefore, by the Kalton-Weis $\gamma$-multiplier theorem (see \cite[Proposition 4.11]{KalW04} and \cite[Theorem 5.2]{NCMA})
\begin{equation}\label{eq:gammamultiplier}
\|I_{S_u}[\Phi_{1,j}(u) f]\|_{\gamma(\R_+,\frac{du}{u};Y_1)}\leq C_T\|\Phi_{1,j}(u) f\|_{\gamma(\R_+,\frac{du}{u};Y_1)}.
\end{equation}
Here we used that $X$ does not contain an isomorphic copy of $c_0$ as it has finite cotype (see \ref{as:Hinfty}).
The remaining two square function norms can be estimated by the square function estimates of Kalton and Weis. Indeed, by \ref{as:Hinfty} and \cite[Theorem 4.11]{HaHa} or \cite[Section 7]{KalW04} (here we again use the finite cotype of $X$)  and the $\gamma$-Fubini property (see \cite[Theorem 13.6]{NCMA}), we obtain
\begin{equation}\label{eq:squarefH}
\begin{aligned}
\|\Phi_{1,j}(u) f\|_{\gamma(\R_+,\frac{du}{u};Y_1)} & \eqsim \|\Phi_{1,j}(u) f\|_{L^p(\R,v;\gamma(\R_+,\frac{du}{u};X))}  \leq C_{A_1} \|f\|_{Y_1},
\\ \|\Phi_{0,j}(u)^\# g\|_{\gamma(\R_+,\frac{du}{u};Y_1)^*} & \eqsim \|\Phi_{0,j}(u)^\# g\|_{L^{p'}(\R,v'; \gamma(\R_+,\frac{du}{u};X)^*)} \leq C_{A_0} \|g\|_{Y_2}.
\end{aligned}
\end{equation}
Combining all the estimates yields
\[|\lb  I_K f, g\rb_{Y_1,Y_2}| \leq C_T C_{A_0} C_{A_1} \|f\|_{Y_1} \|g\|_{Y_2}.\]
Taking the supremum over all $g\in L^{p'}(\R,v';X^\#)$ with $\|g\|_{Y_2}\leq 1$ we find $\|I_K f\|_{Y_1}\leq C_T C_{A_0} C_{A_1} \|f\|_{Y_1}$. This proves the $L^p$-boundedness.
\end{proof}

\begin{remark}\label{rem:hormander}
One can also apply standard extrapolation techniques to obtain weighted boundedness results for singular integrals from the unweighted case (see \cite{chill2014singular, HaHy}). However, for this one needs H\"ormander conditions on the kernel. As our proof gives a result in the more general setting, we can avoid smoothness assumptions on the kernel.
\end{remark}

\section{Maximal $L^p$-regularity\label{sec:maxLp}}

In this section we will apply Theorem \ref{thm:singular} to obtain maximal $L^p$-regularity for the following evolution equation on a Banach space $X_0$.
\begin{equation}\label{eq:Cauchy}
\begin{aligned}
u'(t)+A(t)u(t)& =f(t),\ t\in (0,T)\\
u(0)& =x.
\end{aligned}
\end{equation}
As explained in the introduction no abstract $L^p$-theory is available for \eqref{eq:Cauchy} outside the case where $t\mapsto A(t)$ is continuous.

The following assumption will be made throughout this whole section.
\let\ALTERWERTA\theenumi
\let\ALTERWERTB\labelenumi
\def\theenumi{(A)}
\def\labelenumi{\textbf{(A)}}
\begin{enumerate}
\item\label{as:operatorA} Let $X_0$ be a Banach space and assume the Banach space $X_1$ embeds densely and continuously in $X_0$. Let $p\in [1,\infty)$ and $v\in A_p$ with the convention that $v\equiv 1$ if $p=1$. Let $A:\R\rightarrow \calL(X_{1},X_{0})$ be such that for all $x\in X_1$, $t\mapsto A(t) x$ is strongly measurable, and there is a constant $C>0$ such that
    \[C^{-1} \|x\|_{X_{1}} \leq \|x\|_{X_{0}}+\|A(t)x\|_{X_{0}}\leq C\|x\|_{X_1}.\]
\end{enumerate}
\let\theenumi\ALTERWERTA
\let\labelenumi\ALTERWERTB

The above implies that each $A(t)$ is a closed operator on $X_0$ with $D(A(t)) = X_1$.
Note that whenever $A$ is given on an interval $I\subseteq \R$, we may always extend it constantly or periodically to all of $\R$.

Before we state the main result we will present some preliminary results on evolution equations with time-dependent $A$.

\subsection{Preliminaries on evolution equations}
Evolution equations and evolution families are extensively studied in the literature (see \cite{AT2, EN, Lun, Pazy, Schn, Ta1, Ta2, Ya}). We explain some parts which are different in our set-up.

For a strongly measurable function $f:(a,b)\to X_{0}$ we consider:
\begin{equation}\label{eqprobabstinvalprel}
\begin{cases}
u'(t)+A(t)u(t)=f(t),\ t\in (a,b)\\
u(a)=x,
\end{cases}
\end{equation}
where $u(a)=x$ is omitted if $a=-\infty$.
\begin{enumerate}
\item Assume $-\infty<a<b<\infty$. The function $u$ is said to be a {\em strong solution} of \eqref{eqprobabstinvalprel} if $u\in W^{1,1}(a,b;X_0)\cap L^{1}(a,b;X_1)\cap C([a,b];X_0)$, $u(a) = x$ and \eqref{eqprobabstinvalprel} holds for almost all $t\in (a,b)$.
\item Assume $a=-\infty$ and $b<\infty$. The function $u$ is said to be a {\em strong solution} of \eqref{eqprobabstinvalprel} if $u\in W^{1,1}_{\text{loc}}(a,b;X_0)\cap L^{1}_{\text{loc}}(a,b;X_1)\cap C((a,b];X_0)$ and  $\lim_{s\to a} u(s) = 0$ and \eqref{eqprobabstinvalprel} holds for almost all $t\in (a,b)$.
\item Assume $b=\infty$. The function $u$ is said to be a {\em strong solution} of \eqref{eqprobabstinvalprel} if for every $T>a$ the restriction to $[a,T]$ or $(a,T]$ yield strong solutions in the sense of (1) and (2) respectively.
\end{enumerate}

Note the following simple embedding result for general $A_p$-weights.
\begin{lemma}\label{lem:embedding}
Let $p\in [1, \infty)$ and let $v\in A_p$, where $v\equiv1$ if $p=1$.
For $-\infty<a<b<\infty$,
$W^{1,p}((a,b),v;X_0) \hookrightarrow C([a,b];X_0)$ and
\begin{align*}
\|u\|_{C([a,b];X_0)} \leq C\|u\|_{W^{1,p}((a,b),v;X_0)}.
\end{align*}
\end{lemma}

\begin{proof}
Since $L^p((a,b),v;X_0)\hookrightarrow L^1(a,b;X_0)$, and $u(t) - u(s) =  \int_s^t u'(r) \, dr$, the continuity of $u$ is immediate. Moreover,
\begin{align*}
\|u(t)\| \leq \|u(s)\| + \int_s^t \|u'(r)\| \, dr\leq \|u(s)\| + C\|u'\|_{L^p((a,b),v;X_0)}.
\end{align*}
Taking $L^p((a,b),v)$-norms with respect to the $s$-variable yields the result.
\end{proof}

There is a correspondence between the evolution problem \eqref{eqprobabstinvalprel} and evolution families as defined below.
\begin{definition}\label{def:evolut}
Let $(A(t))_{t\in \R}$ be as in \ref{as:operatorA}. A two parameter family of bounded linear operators $S(t,s)$, $s\leq t$, on a Banach space $X_{0}$ is called an \textit{evolution system for $A$} if the following conditions are satisfied:
\begin{itemize}
  \item[(i)] {$S(s,s)=I,\ S(t,r)S(r,s)=S(t,s)$ for $s\leq r\leq t$;}
  \item[(ii)]{$(t,s)\rightarrow S(t,s)$ is strongly continuous for $s\leq t$.}
  \item[(iii)] For all $s\in \R$ and $T\in (s, \infty)$, for all $x\in X_1$, the function $u:[s,T]\to X_0$ defined by $u(t) = S(t,s)x$ is in $L^1(s,T;X_1)\cap W^{1,1}(s,T;X_0)$ and satisfies $u'(t) + A(t) S(t,s)x=0$ for almost all $t\in(s,T)$.
  \item[(iv)] For all $t\in \R$ and $T\in (-\infty, t]$ for all $x\in X_1$, the function $u:[T,t]\to X_0$ defined by $u(s) = S(t,s)x$ is in $L^{1}(T,t;X_1)\cap W^{1,1}(T,t;X_0)$ and satisfies $u'(s) = S(t,s) A(s) x$.
\end{itemize}
\end{definition}
Note that (iii) says that $u$ is a strong solution of \eqref{eqprobabstinvalprel} with $f=0$.

\begin{example} \
If $A(t)=A$ is independent of $t$ and sectorial of angle $<\pi/2$, then $S(t,s)=e^{-(t-s)A}$ and the two-parameter family of operators reduces to the one-parameter family $e^{-tA}$, $t\geq 0$, which is the semigroup generated by $-A$.
\end{example}

\begin{example}\label{ex:evolution}
Assume $A:\R\to \calL(X_1, X_0)$ is strongly measurable and satisfies \ref{as:operatorA}.
Define a family of operators $\calA$ by
\[\calA = \{A(t): t\in \R\} \cup \Big\{\frac{1}{t-s}\int_{s}^{t} A(r) \, dr: s<t\Big\}.\]
Here we use the strong operator topology to define the integral. Assume there exist $\phi$, $M$ and $N$ such that all $B\in \calA$ are all sectorial of angle $\phi<\pi/2$ and for all $\lambda\in \Sigma_{\phi}$, \
\[\|\lambda (\lambda + B)^{-1}\|\leq M \ \ \  \text{and} \ \ \|x\|_{X_1}\leq N(\|x\|_{X_0} + \|Bx\|_{X_0}) \]
Assume for every $B_1,B_2\in \calA$ and $\lambda,\mu\in \Sigma_{\phi}$, the operators $(\lambda + B_1)^{-1}$ and $(\mu + B_2)^{-1}$ commute.
Define $S(t,s) = e^{-(t-s)A_{st}}$, where $A_{st} = \frac{1}{t-s}\int_s^t A(r) \, dr$.
Then $S$ is an evolution family for $A$. Here the exponential operator is defined by the usual Cauchy integral (see \cite[Chapter 2]{Lun}). Usually, no simple formula for $S$ is available if the operators in $\calA$ do not commute.

Note that in this special case the kernel $K(t,s) = \one_{\{s<t\}} A(0)e^{-\lambda(t-s)} S(t,s)$ satisfies the Calder\'on-Zygmund estimates of \cite{HaHy}. Indeed, note that $\frac{\partial K}{\partial t} = -\one_{\{s<t\}}(\lambda+A(t)) A(0) e^{-\lambda(t-s)}S(t,s)$ and $\frac{\partial K}{\partial s} = \one_{\{s<t\}}(\lambda+A(s)) A(0) e^{-\lambda(t-s)}S(t,s)$. Now since for all $r\in \R$ and $B\in \calA$, $\|A(r) x\|\leq NC (\|x\|_{X_0} + \|Bx\|_{X_0})$, we find that for all $r,\tau\in \R$ and $s<t$ letting $\sigma = (t+s)/2$,
\begin{align*}
\|A(r) A(\tau) S(t,s)\| & = \|A(r) S(t,\sigma)\| \,\|A(\tau) S(\sigma,s)\|
\\ & \leq N^2C^2 (1+ \|A_{\sigma t} S(t,\sigma)\|)(1+ \|A_{s \sigma} S(\sigma,s)\|)
\\ & \leq C'(1+(t-s)^{-1})^2 \leq \tfrac{3}{2}C'(1+(t-s)^{-2}).
\end{align*}
Therefore, the extrapolation results from the unweighted case to the weighted case of Remark \ref{rem:hormander} does hold in this situation.
\end{example}

\begin{proposition}\label{prop:evolutrepr}
Let $S$ be an evolution family for $A$.
Fix $x\in X_0$ and $f\in L^1(a,b;X_0)$. If \eqref{eqprobabstinvalprel} has a strong solution $u\in L^{1}(a,b;X_1)\cap W^{1,1}(a,b;X_0)\cap C([a,b];X_0)$, then it satisfies
\begin{equation}\label{eq:mild}
u(t) = S(t,s)u(s) + \int_s^t S(t,r) f(r) \, dr,  \ \ a<s\leq t<b,
\end{equation}
where we allow $s=a$ and $t = b$ whenever these are finite numbers.
In particular, strong solutions are unique if $a>-\infty$. In the case $a=-\infty$ this remains true if $\lim_{s\to -\infty} \|S(t,s)\| = 0$.
\end{proposition}
A partial converse is used and proved in Theorem \ref{teostep1generalized}.

\begin{proof}
Fix $a<s<t<b$.
By approximation one easily checks that for $u\in W^{1,1}(s,b;X_0)\cap L^1(s,b;X_1)$, $r\mapsto S(t,r) u(r)$ is in $W^{1,1}(s,b;X_0)$ and
\begin{equation}\label{eq:product}
\frac{d}{dr} [S(t,r) u(r)] = -S(t,r) A(r) u(r) + S(t,r) u'(r), \ \ \ r\in (s,T).
\end{equation}
Applying \eqref{eq:product} to the strong solution $u$ of \eqref{eqprobabstinvalprel}, yields
$\frac{d}{dr} [S(t,r) u(r)]  = S(t,r) f(r)$.
Integrating this identity over $(s,t)$, we find  \eqref{eq:mild}.

If $a>-\infty$, then we may take $s=a$ in the above proof and hence we can replace $u(s) = u(a)$ by the initial value $x$.
If $a=\infty$, the additional assumption on $S$ allows us to let $s\to -\infty$ to obtain
\[u(t) = \int_{-\infty}^t S(t,r) f(r) \, dr,  \ \ t<b.\]
\end{proof}

\begin{corollary}\label{cor:uniqueS}
If $S_1$ and $S_2$ are both evolution families for $A$, then $S_1 = S_2$.
\end{corollary}

\subsection{Assumptions on $A$}

The following condition can be interpreted as an abstract ellipticity condition.
\let\ALTERWERTA\theenumi
\let\ALTERWERTB\labelenumi
\def\theenumi{(E)}
\def\labelenumi{\textbf{(E)}}
\begin{enumerate}
\item\label{as:Evol} Assume that $X_0$ has finite cotype and assume that there exists $A_0\in \calL(X_1, X_0)$ which has a bounded $H^\infty$-calculus of angle $\sigma<\pi/2$ and there exists a strongly continuous evolution system $(T(t,s))_{s\leq t}$ for $(A(t)-A_{0})_{t\in \R}$
such that $e^{-rA_{0}}$ commutes with $T(t,s)$ for every $t\geq s$ and $r\in \R_{+}$ and
assume there exists an $\omega\in\R$ such that
\begin{equation}
\|T(t,s)\|_{\calL(X_{0})}\leq Me^{\omega(t-s)}, \ \ \ s\leq t.
\nonumber
\end{equation}
\end{enumerate}
\let\theenumi\ALTERWERTA
\let\labelenumi\ALTERWERTB
Set $T(t,s) =0 $ for $t<s$. The following $\Rr$-boundedness condition will be used.
\let\ALTERWERTA\theenumi
\let\ALTERWERTB\labelenumi
\def\theenumi{(Rbdd)}
\def\labelenumi{\textbf{(Rbdd)}}
\begin{enumerate}
\item\label{as:RbddT} Assume that the family $\mathcal{I}:=\{I_{\omega,kT}:\ k\in \calK\}\subseteq\calL(L^{p}(\R,v,X_{0}))$ is $\Rr$-bounded, where for $k\in \mathscr{K}$ and $f\in L^{p}(\R,v;X_{0})$,
\begin{equation}
I_{\omega, kT}f(t):=\int_{\R}k(t-s)e^{-\omega|t-s|}T(t,s)f(s)ds.
\nonumber
\end{equation}
\end{enumerate}
\let\theenumi\ALTERWERTA
\let\labelenumi\ALTERWERTB

\begin{remark}
\
\begin{enumerate}
\item By \ref{as:operatorA} and \ref{as:Evol} there is a constant $C$ such that
\begin{equation}\label{eq:equivalnorms}
\begin{aligned}
C^{-1}(\|A(t)x\|_{X_0}+\|x\|_{X_0}) & \leq \|A_0 x\|_{X_{0}} +\|x\|_{X_0}
\\ & \leq C(\|A(t)x\|_{X_0}+\|x\|_{X_0}), \ \ \  t\in \R
\end{aligned}
\end{equation}
and both norms are equivalent to $\|x\|_{X_1}$.

\item For $m$ even, if the $A(t)$ are $m$-th order elliptic operators with $x$-independent coefficients one typically takes $A_0 =  \delta (-\Delta)^m$ with $\delta>0$ small enough.
\item For $p,q\in (1, \infty)$, $v\in A_p$ and $X = L^q$, the $\Rr$-boundedness assumption follows from the weighted boundedness of $T(t,s)$ for all $w\in A_q$ (see Theorem \ref{thm:weightedR}).
\item Although we allow $p=1$ and $v=1$ in the above assumptions, checking the assumption \ref{as:RbddT} seems more difficult in this limiting case.
\end{enumerate}
\end{remark}

\begin{lemma}\label{lem:evfam}
Under the assumptions \ref{as:operatorA} and \ref{as:Evol} the evolution family $S$ for $A$ uniquely exists and satisfies
\begin{equation}\label{eq:Sfactor}
\begin{aligned}
S(t,s) & =e^{-(t-s)A_0} T(t,s)
\\ & = T(t,s)e^{-(t-s)A_0}  =e^{-\frac12(t-s)A_0} T(t,s)e^{-\frac12(t-s)A_0}, \ \ s\leq t,
\end{aligned}
\end{equation}
and there is a constant $C$ such that for all $s\leq t$, $\|S(t,s)\|_{\calL(X_0)}\leq Ce^{\omega(t-s)}$.
Moreover, there is a constant $C$ such that,
\[\|S(t,s)\|_{\calL(X_1)} \leq C\|S(t,s)\|_{\calL(X_0)},  \ \  s\leq t.\]
\end{lemma}
\begin{proof}
The second identity follows from \ref{as:Evol}. To prove the first identity, we check that $S(t,s)$ given by \eqref{eq:Sfactor} is an evolution family for $A$. By Corollary \ref{cor:uniqueS} this would complete the proof. It is simple to check properties (i) and (ii) of Definition \ref{def:evolut} and it remains to check (iii) and (iv). Let $x\in X_1$. By the product rule for weak derivatives and \ref{as:Evol} we find
\begin{align*}
\frac{d}{dt} S(t,s) x &= - A_0 e^{-(t-s)A_0} T(t,s)x - (A(t) - A_0) T(t,s)e^{-(t-s)A_0}x
\\ & = - A_0S(t,s)x - (A(t)-A_0) S(t,s)x = -A(t)S(t,s)x.
\end{align*}
Similarly, one checks that $\frac{d}{ds} S(t,s) x = S(t,s)A(s) x$. The fact that $S(t,s)$ satisfies the same exponential estimate as $T(t,s)$ follows from the estimate \eqref{eq:semigroupbound} applied to $A_0$.

By assumptions, for every $x\in X_{1}$, $e^{-rA_{0}}S(t,s)x=S(t,s)e^{-rA_{0}}x$. Thus, by differentiation we find $-A_{0}S(t,s)x=-S(t,s)A_{0}x$ and therefore
\begin{align*}
\|S(t,s)x\|_{X_{1}}&\leq C(\|A_{0}S(t,s)x\|_{X_{0}}+\|S(t,s)x\|_{X_{0}})
\\ & \leq C(\|S(t,s)A_{0}x\|_{X_{0}}+\|S(t,s)x\|_{X_{0}})
\\ &\leq C\|S(t,s)\|_{\calL(X_0)} (\|A_{0}x\|_{X_{0}}+\|x\|_{X_{0}})\leq C'\|S(t,s)\|_{\calL(X_0)} \|x\|_{X_{1}}.
\end{align*}
\end{proof}

\subsection{Main result on maximal $L^p$-regularity}

Next we will present our main abstract result on the regularity of the strong solution to the problem
\begin{equation}\label{problemteoA0inf}
u'(t)+(A(t)+\lambda)u(t)=f(t),\ \ t\in\R.\\
\end{equation}

\begin{theorem}\label{teostep1generalized}
Assume \ref{as:operatorA}, \ref{as:Evol}, and \ref{as:RbddT}. For any $\lambda>\omega$ and for every $f\in L^{p}(\R,v;X_{0})$ there exists a unique strong solution $u\in W^{1,p}(\R,v;X_{0})\cap L^{p}(\R,v;X_{1})$ of \eqref{problemteoA0inf}.
Moreover, there is a constant $C$ independent of $f$ and $\lambda$ such that
\begin{equation}\label{eqaprioriA0gen}
\begin{aligned}
(\lambda-\omega)\|u\|_{L^{p}(\R,v,X_{0})}+\|A_0 u\|_{L^{p}(\R,v;X_{0})}& \leq C\|f\|_{L^{p}(\R,v;X_{0})}
\\  \|u'\|_{L^{p}(\R,v;X_{0})} & \leq \tfrac{C(\lambda-\omega+1)}{\lambda-\omega}  \|f\|_{L^{p}(\R,v;X_{0})}.
\end{aligned}
\end{equation}
\end{theorem}

\begin{remark}\label{rem:uprimeest}
Parts of the theorem can be extended to $\lambda= \omega$, but we will not consider this in detail.
The constant in the estimate \eqref{eqaprioriA0gen} for $u'$ can be improved if one knows $\|A(t) x\|_{X_0}\leq C\|A_0x\|_{X_0}$ or when taking $\lambda \geq \omega+1$ for instance.
\end{remark}

Before we turn to the proof of Theorem \ref{teostep1generalized} we introduce some shorthand notation.
Let $S_{\lambda}(t,s) = e^{-\lambda(t-s)}S(t,s)$ and $T_{\lambda}(t,s) = e^{-\lambda(t-s)}T(t,s)$. Since by Lemma \ref{lem:evfam}, $S$ is an evolution family for $A$, also $S_{\lambda}$ is the evolution family for $A(t) + \lambda$. Similarly, $T_{\lambda}(t,s)$ is an evolution family for $A(t) - A_0+ \lambda$.
By \eqref{eq:mild} if the support of $f\in L^1(\R;X_0)$ is finite, a strong solution of \eqref{problemteoA0inf} satisfies
\begin{equation}\label{eq:strongsolutionabstractdef}
u(t)=\int_{-\infty}^{t}S_{\lambda}(t,r)f(r) \, dr,  \ \ \ t\in \R.
\end{equation}

\begin{proof}

Replacing $A(t)$ and $T(t,s)$ by $A(t) + \omega$ and $e^{-(t-s)\omega}T(t,s)$ one sees that without loss of generality we may assume $\omega = 0$ in \ref{as:Evol} and \ref{as:RbddT}. We first prove that $u$ given by \eqref{eq:strongsolutionabstractdef}, is a strong solution and \eqref{eqaprioriA0gen} holds. First let $f\in L^p(\R,v;X_{1})$ and such that $f$ has support on the finite interval $[a,b]$. Later on we use a density argument for general $f\in L^p(\R,v;X_0)$. Let $u$ be defined as in \eqref{eq:strongsolutionabstractdef}. Note that $u = 0$ on $(-\infty, a]$.

\medskip

{\em Step 1}: By Lemma \ref{lem:evfam} the function $u$ defined by  \eqref{eq:strongsolutionabstractdef} satisfies
\begin{align*}
\|u(t)\|_{X_1} & \leq \int_{-\infty}^{t} \|S_{\lambda}(t,s)\|_{\calL(X_1)} \|f(s)\|_{X_1} \, ds
\\ & \leq C' \|f\|_{L^1(a,b;X_1)}\leq C([v]_{A_p}) \|f\|_{L^p(\R,v;X_1)}.
\end{align*}
We show that $u$ is a strong solution of \eqref{eqprobabstinvalprel}.
Observe that from Fubini's Theorem and $\frac{d}{ds}S_{\lambda}(s,r) x = -(\lambda+A(s))S_{\lambda}(s,r)x$ for $x\in X_1$, we deduce
\begin{align*}
\int_{-\infty}^{t}(\lambda+A(s))u(s)\, ds&=\int_{-\infty}^{t}\int_{-\infty}^{s}(\lambda+A(s))S_{\lambda}(s,r)f(r)\, dr\, ds
\\ & =\int_{-\infty}^{t}\int_{r}^{t}(\lambda+A(s))S_{\lambda}(s,r)f(r)\, ds\, dr
\\ &=\int_{-\infty}^{t}(-S_{\lambda}(t,r)f(r)+f(r))dr=-u(t)+\int_{-\infty}^{t}f(r)dr.
\end{align*}
Therefore, $u$ is a strong solution of \eqref{problemteoA0inf}.

\medskip

{\em Step 2}: In this step we show there exists a $C\geq 0$ independent of $\lambda$ and $f$ such that
\begin{equation}\label{eq:maxregest1}
\|A_0u\|_{L^{p}(\R,v;X_{0})} \leq C\|f\|_{L^{p}(\R,v;X_{0})}.
\end{equation}
By \eqref{eq:Sfactor} and \eqref{eq:strongsolutionabstractdef} we can write $A_0 u = I_Kf$, where
\[K(t,s) = \frac{\phi((t-s)A_0) T_{\lambda}(t,s) \phi((t-s)A_0)}{t-s}.\]
Here $\phi\in H^\infty_{0}(\Sigma_{\sigma'})$ for $\sigma'<\pi/2$ is given by  $\phi(z) = z^{1/2} e^{-z/2}$. In order to apply Theorem \ref{thm:singular}, we note that all assumptions \ref{as:HXpv}-\ref{as:Rbdd} are satisfied. Only the $\Rr$-boundedness condition \ref{as:Rbdd} requires some comment. Note that $k\in \calK$ implies that for all $\lambda\geq 0$, $k_{\lambda}\in \calK$ where $k_{\lambda}(t) = e^{-\lambda t}\one_{\{t>0\}} k(t)$. Therefore, it follows from \ref{as:RbddT} that for all $\lambda\geq 0$,
\[\Rr(I_{k T_{\lambda}}: k\in \calK) = \Rr(I_{k_{\lambda} T}: k\in \calK) \leq \Rr(I_{k T}: k\in \calK)<\infty\]
which gives \ref{as:Rbdd} with a uniform estimate in $\lambda$. Now \eqref{eq:maxregest1} follows from Theorem \ref{thm:singular}.

\medskip

{\em Step 3}: In this step we show there exists a $C\geq 0$ independent of $\lambda$ and $f$ such that
\begin{equation}\label{eq:maxregest2}
\lambda \|u\|_{L^{p}(\R,v;X_{0})} \leq C\|f\|_{L^{p}(\R,v;X_{0})}.
\end{equation}
Using \eqref{eq:strongsolutionabstractdef} and $\|S(t,s)\|\leq C$ we find
\begin{align*}
\lambda \|u\|_{X_{0}}&\leq \int_{-\infty}^{t} \lambda\|S_{\lambda}(t,s)f(s)\|_{X_0} \, ds \leq C\int_{-\infty}^{t} \lambda e^{-\lambda(t-s)} \|f(s)\|_{X_{0}}ds
 \leq  Cr_{\lambda} * g(t),
\end{align*}
where $r_{\lambda}(t) = \lambda e^{-\lambda|t|}$ and $g(s) = \|f(s)\|_{X_{0}}$
As $r_1\in L^1(\R)$ is radially decreasing by \cite[Theorem 2.1.10]{GrafakosClassical} and
\cite[Theorem 9.1.9]{GrafakosModern},
\begin{align*}
\lambda\|u\|_{L^{p}(\R_{+},v;X_{0})} & \leq C\|r_{\lambda} * g\|_{L^{p}(\R,v)} \\ & \leq   C \|Mg\|_{L^{p}(\R,v)} \leq C'\|g\|_{L^{p}(\R,v)} = C'\|f\|_{L^{p}(\R,v;X_{0})}
\end{align*}
in the case $p>1$. The case $p=1$ follows from Fubini's theorem and the convention $v\equiv 1$.
This estimate yields \eqref{eq:maxregest2}.

\medskip

{\em  Step 4:} To prove the estimate for $u'$ note that $u' = -\lambda u -A u+f$, and hence
writing $Z = L^{p}(\R,v;X_{0})$, by \eqref{eq:equivalnorms} and \eqref{eqaprioriA0gen}, we obtain
\begin{align*}
\|u'\|_{Z}& \leq \lambda \|u\|_{Z}  +\|A u\|_{Z} +\|f\|_{Z}
\\ & \leq (\lambda + C) \|u\|_{Z}  + C\|A_0 u\|_{Z} +\|f\|_{Z}
\leq K \Big(\frac{\lambda + C}{\lambda-\omega} + 1\Big)  \|f\|_{Z}.
\end{align*}
This finishes the proof of \eqref{eqaprioriA0gen} for $f\in L^p(\R;X_{1})$ with support in $[a,b]$

\medskip

{\em Step 5}: Now let $f\in L^p(\R,v;X_0)$. Choose for $n\geq 1$,  $f_n\in L^p(\R,v;X_{1})$ with compact support and such that $f_n\to f$ in $L^p(\R,v;X_{0})$. For each $n\geq 1$ let $u_n$ be the corresponding strong solution of \eqref{problemteoA0inf} with $f$ replaced by $f_n$.
From \eqref{eqaprioriA0gen} applied to $u_n - u_m$ we can deduce that $(u_n)_{n\geq 1}$ is a Cauchy sequence and hence convergent to some $\overline{u}$ in $L^p(\R,v;X_{1})\cap W^{1,p}(\R,v;X_0)$.  On the other hand, for $u$ defined as in \eqref{eq:strongsolutionabstractdef} one can show in the same way as in Step 3 that for almost all $t\in \R$,
\begin{align*}
\|u(t) - u_n(t)\|& \leq \int_{-\infty}^t \|S_{\lambda}(t,s)\| \, \|f(s) - f_n(s)\| \,ds
\\ & \leq C \int_{-\infty}^t e^{-\lambda (t-s)}\|f(s) - f_n(s)\| \,ds
\leq C\ M( \|f-f_n\|)(t),
\end{align*}
where $M$ is the Hardy-Littlewood maximal operator.
Taking $L^p(v)$-norms and using the boundedness of the maximal operator we find $u_n\to u$ in $L^p(\R,v;X_0)$ and hence $u=\overline{u}$ if $p\in (1, \infty)$. Taking limits (along a subsequence), \eqref{problemteoA0inf} and \eqref{eqaprioriA0gen} follow if $p\in (1, \infty)$. The case $p=1$ is proved similarly using Young's inequality.
\end{proof}

It will be convenient to restate our results in terms of maximal $L^p_v$-regularity. For $-\infty\leq a<b\leq \infty$, let
\[\MR = W^{1,p}((a,b),v;X_0)\cap L^{p}((a,b),v;X_{1}).\]

\begin{definition}\label{def:MR}
Let $-\infty\leq a<b\leq \infty$. Assume \ref{as:operatorA} holds and let $p\in [1, \infty)$ and $v\in A_p$ with the convention that $v\equiv 1$ if $p=1$. The operator-valued function $A$ is said to have {\em maximal $L^p_v$-regularity on $(a,b)$} if for all $f\in L^p((a,b),v;X_0)$, the problem
\begin{equation}\label{eq:generalMR}
\begin{cases}
u'(t)+A(t)u(t)=f(t),\ \ t\in(a,b)\\
u(a)=0,
\end{cases}
\end{equation}
has a unique strong solution $u:(a,b)\to X_0$ and there is a constant $C$ independent of $f$ such that
\begin{equation}\label{eq:MRest}
\|u\|_{\MR}\leq C\|f\|_{L^{p}((a,b),v;X_{0})}.
\end{equation}
Here we omit the condition $u(a) = 0$ if $a=-\infty$.
\end{definition}
Of course, the reverse estimate of \eqref{eq:MRest} holds trivially. Note that maximal $L^p_v$-regularity on $(a,b)$ implies maximal $L^p_v$-regularity on $(c,d)\subseteq (a,b)$.
It is also easy to check that if $|b-a|<\infty$, the maximal $L^p_v$-regularity on $(a,b)$ for $A$ and $\lambda+A$ are equivalent. Indeed, the solutions of $u'(t) + (\lambda+A(t)) u(t) = f(t)$ and $w'(t) + A(t) w(t) = e^{\lambda t}f(t)$ are connected by the identity $u(t) = e^{-\lambda t} w(t)$.

The result of  Theorem \ref{teostep1generalized} immediately implies that
\begin{corollary}
Assume \ref{as:operatorA}, \ref{as:Evol} and \ref{as:RbddT}.
For any $\lambda>\omega$, $\lambda+A$ has maximal $L^p_v$-regularity on $\R$.
\end{corollary}
Actually the constant in the estimate can be taken uniformly in $\lambda$. Indeed, for fixed $\lambda_0>\omega$ by \eqref{eqaprioriA0gen} and Remark \ref{rem:uprimeest}, there is a constant $C$ such that for all $\lambda\geq \lambda_0$ and for all $f\in L^p(\R_+,v;X_0)$,
 \begin{equation}\label{MReqaprioriA0gen}
\|u\|_{\MRinfR}\leq C\|f\|_{L^{p}(\R,v;X_{0})}.
\end{equation}
This is a maximal regularity estimate with constant which is uniform in $\lambda$.

\begin{remark}
If $A$ is time independent and has an $H^\infty$-calculus of angle $<\pi/2$, then setting $A_0 = A$, and $T(t,s) = I$, Theorem \ref{teostep1generalized} yields a maximal regularity result for autonomous equations. There are much more suitable ways to derive maximal $L^p$-regularity results in the autonomous case (see \cite{KWcalc,KW,Weis01,Weis}), using less properties of the operator $A$. Indeed, only $\Rr$-sectoriality of $A$ is needed, but the Banach space $X_0$ is assumed to be a UMD space.
We assume more on the operator but less on the space as we only require finite cotype of $X_0$ and the $\Rr$-boundedness of a certain integral operator. Another theory where no assumptions on the Banach space are made but even more on the operator, can be found in \cite{KalKuc}.
In the above mentioned works only maximal $L^p$-regularity on $\R_+$ is considered, but by a standard trick due to Kato one can always reduce to this case (see for instance the proof of \cite[Theorem 7.1]{Dore}). For the case of time-dependent operators this is no longer true.
\end{remark}

\subsection{Traces and initial values\label{subsec:tracesinitial}}

Recall from Lemma \ref{lem:embedding} that any $u\in W^{1,p}((a,b),v;X_0)$ has a continuous version.
We introduce certain interpolation spaces in order to give a more precise description of traces. Let $X_{v,p}$ be the space of all $x\in X_0$ for which there is a $u\in \MRinfplus$ such that $u(0)=x$. Let
\begin{equation}\label{eq:Xvp}
\|x\|_{X_{v,p}} = \inf\{\|u\|_{\MRinfplus}: u(0)=x\}.
\end{equation}
Spaces of this type have been studied in the literature (see \cite{BMR, BK91, Kalj} and references therein).
Obviously, one has $X_1\hookrightarrow X_{v,p}\hookrightarrow X_0$.

For $t\in \R$ and a weight $v$, let $v_t = v(\cdot-t)$. The following trace estimate on $\R_+$ is a direct consequence of the definitions. A similar assertions holds for $u\in \MRinfR$ for all $t\in \R$.
\begin{proposition}[Trace estimate]\label{prop:traceest}
For $u\in \MRinfplus$, one has
\[\|u(t)\|_{X_{v_t,p}} \leq \|u\|_{\MRinfplus},  \ \ t\in [0,\infty).\]
\end{proposition}

A simple application of maximal regularity is that one can automatically consider nonzero initial values. Note that without loss of generality we can let $a=0$.
\begin{proposition}\label{prop:initialvalue}
Assume \ref{as:operatorA} and let $T\in (0,\infty]$. Assume $A$ has maximal $L^p_v$-regularity on $(0,T)$ with constant $K_A$. For $x\in X_0$ and $f:(0,T)\to X_0$ strongly measurable the following are equivalent:
\begin{enumerate}[$(1)$]
\item The data satisfies $x\in X_{v, p}$ and $f\in L^p((0,T),v;X_0)$
\item There exists a unique strong solution $u\in \MRT$ of
\begin{equation}\label{eq:cauchyinitial}
\begin{cases}
u'(t)+A(t)u(t)=f(t),\ \ t\in (0,T)\\
u(0)=x.
\end{cases}
\end{equation}
\end{enumerate}
In this case there is a constant $c_{v,p,T}$ such that the following estimate holds:
\begin{equation}\label{eqaprioriA0genInitial}
\begin{aligned}
\max\{c_{v,p,T} \|x\|_{X_{v,p}}, \|f\|_{L^{p}((0,T),v;X_{0})}\} & \leq \|u\|_{\MRT} \\ & \leq K_A \|x\|_{X_{v,p}}+ K_A\|f\|_{L^{p}((0,T),v;X_{0})}.
\end{aligned}
\end{equation}
\end{proposition}
\begin{proof}
(1) $\Rightarrow$ (2): \ Let $w\in \MRinfplus$ be such that $w(0) = x$.
Let $g(t) = -(w'(t)+A(t)w(t))$. Then $g\in L^p((0,T),v;X_0)$. Let $\tilde{u}$ be the solution to \eqref{eq:cauchyinitial} with zero initial value and with $f$ replaced by $f+g$. Now $u(t) = \tilde{u}(t) + w(t)$ is the required strong solution of \eqref{eq:cauchyinitial}.
Indeed, clearly $u(0) = x$ and
\begin{align*}
u'(t) + A(t) u(t) & = \tilde{u}'(t)  + A(t) \tilde{u}(t) + w'(t) + A(t) w(t)
\\ &  = f+ g -g = f.
\end{align*}
Moreover,
\begin{align*}
&\|u\|_{\MRT}  \leq  \|\tilde{u}\|_{\MRT} + \|w\|_{\MRT} \\ & \leq K_A \|f\|_{L^p((0,T),v;X_0)} + K_A\|w\|_{\MRinfplus}.
\end{align*}
Taking the infimum over all $w\in \MRinfplus$ with $w(0)=x$ also yields the second part of \eqref{eqaprioriA0genInitial}.

(2) $\Rightarrow$ (1): \ As $u'$ and $Au$ are both in $L^p((0,T),v;X_0)$, the identity in \eqref{eq:cauchyinitial} yields that $f\in L^p((0,T),v;X_0)$ with the estimate as stated. To obtain the required properties for $x$ note that $u\in \MRT$ can be extended to a function $u\in \MRinfplus$ with $c_{v,p,T}\|u\|_{\MRinfplus}\leq \|u\|_{\MRT}$. In the case $T=\infty$ we can take $c_{v, p,T} = 1$.
\end{proof}

It can be difficult to identify $X_{v,p}$. For power weights this is possible. Including a power weight has become an important standard technique to allow non-smooth initial data and to create compactness properties. At the same time, the regularity properties of the solution to \eqref{eq:cauchyinitial} for $t>0$ are unchanged. For more details and applications to evolution equations we refer to
\cite{Grisvard, KPW, Lun, MS12a, MS12b, PS04}.
\begin{example}\label{ex:chtrace}
Assume $v(t) = t^{\alpha}$ with $\alpha\in (-1, p-1)$. Then $v\in A_p$ and
$X_{v,p} = (X_0, X_1)_{1 - \frac{1+\alpha}{p}, p}$ (see \cite[Theorem 1.8.2]{Tr1}). Here $(X_0, X_1)_{\theta, p}$ stands for the real interpolation space between $X_0$ and $X_1$. In the limiting cases $\alpha\uparrow p-1$ and $\alpha\downarrow -1$, one sees that the endpoint $X_1$ and $X_0$ can almost be reached.

As in \cite{PS04} we find that for $\alpha\in [0,p-1)$, any $u\in \MRinfplus$ has a continuous version with values in $(X_0, X_1)_{1 - \frac{1+\alpha}{p}, p}$ and
\begin{equation}\label{eq:tracestalpha}
\sup_{t\in \R_+} \|u(t)\|_{(X_0, X_1)_{1 - \frac{1+\alpha}{p}, p}} \leq C \|u\|_{\MRinfplus}.
\end{equation}
Indeed, this follows from the boundedness and strong continuity of the left-translation in $L^p(\R_+,v;(X_0, X_1)_{1 - \frac{1+\alpha}{p}, p})$ and Proposition \ref{prop:traceest}.

On the other hand, for every $-1<\alpha<p-1$ one has $u\in C((0,\infty); (X_0, X_1)_{1 - \frac{1}{p}, p})$ and for every $\varepsilon>0$,
\[\sup_{t\in [\varepsilon,\infty)} t^{\alpha/p} \|u(t)\|_{(X_0, X_1)_{1 - \frac{1}{p}, p}} \leq C \|t\mapsto t^{\alpha/p} u(t)\|_{\text{MR}^p(\varepsilon, \infty)}
\leq C_{\varepsilon} \| u\|_{\MRinfplus},\]
where we used $t^{-p}\leq \max\{1, \varepsilon^{-p}\}$.  If additionally $u(0) = 0$, then by Hardy's inequality (see \cite[p. 245-246]{HardyLittlePol}) we can take $\varepsilon=0$ in the last estimate.
\end{example}

\begin{proof}[Proof of Theorem \ref{thm:HSintro}]
First of all we may use a constant extension of $A$ to an operator family on $\R$. Clearly, we can do this in such a way that $T(t,s)$ is uniformly bounded in $-\infty<s\leq t<\infty$ say by a constant $M$. For instance one can take $A(t) = A_0$ for $t\notin (0,\tau)$. Assumption \ref{as:operatorA} is clearly satisfied. Note that by the assumption and \cite[Theorem 11.13]{KW}, $A_0$ has a bounded $H^\infty$-calculus of angle $<\pi/2$ and hence \ref{as:Evol} is satisfied.

By Proposition \ref{prop:uniformTts} $\{I_{kT}: k\in \calK\}$ is uniformly bounded. For $p=2$, this implies $\Rr$-boundedness of $\{I_{kT}: k\in \calK\}\subseteq \calL(L^2(\R,v;X_0))$, because $L^2(\R,v;X_0)$ is a Hilbert space. By Proposition \ref{prop:pindRbdd} this implies that $\{I_{kT}: k\in \calK\}\subseteq \calL(L^p(\R,v;X_0))$ is $\Rr$-bounded as well and hence condition \ref{as:RbddT} holds.
Therefore, all the conditions of Theorem \ref{teostep1generalized} are satisfied, and we find that $A$ has maximal $L^p_v$-regularity on $\R$. This implies that $A$ has maximal $L^p_v$-regularity on $(0,\tau)$, and hence the required result follows from Proposition \ref{prop:initialvalue} and Example \ref{ex:chtrace}.
\end{proof}

\subsection{Perturbation and approximation}

In this section we will illustrate how the additional parameter $\lambda$ from \eqref{MReqaprioriA0gen} can be used to solve the perturbed problem
\begin{equation}\label{eqprobperturbation}
\begin{cases}
u'(t)+A(t)u(t)+B(t,u(t))=f(t),\ t\in(0,T)\\
u(0)=x.
\end{cases}
\end{equation}
Here $B:[0,T]\times X_{1}\rightarrow X_{0}$ is such that there exists a constant $\varepsilon>0$ small enough and constants $C,L\geq 0$ such that for all $x,y\in X_1$ and $t\in (0,T)$,
\begin{equation}\label{eq:Blipschitz}
\begin{aligned}
\|B(t,x)-B(t,y)\|_{X_{0}}& \leq\varepsilon\|x-y\|_{X_{1}}+L_B \|x-y\|_{X_{0}},
\\ \|B(t,x)\|_{X_{0}}& \leq C_B(1+\|x\|_{X_{1}}).
\end{aligned}
\end{equation}

Recall that $\MRT = W^{1,p}((0,T),v;X_0) \cap L^{p}((0,T),v;X_{1})$.

\begin{proposition}\label{prop:perturb}
Assume $T<\infty$. Assume \ref{as:operatorA} holds and assume there is a $\lambda_0$ such that for all $\lambda\geq \lambda_0$, $\lambda+A$ has maximal $L^p_v$-regularity on $(0,T)$ and there is a constant $C_A>0$ such that for all $\lambda\geq \lambda_0$ and $f\in L^{p}((0,T),v;X_{0})$, the strong solution $u$ to \eqref{eq:generalMR} satisfies
\begin{equation}\label{eqaprioriA0genTpert}
\lambda\|u\|_{L^p((0,T),v;X_0)} + \|u\|_{\MRT}\leq C_A\|f\|_{L^{p}((0,T),v;X_{0})}.
\end{equation}
Assume the constant from \eqref{eq:Blipschitz} satisfies $\varepsilon<\frac{1}{C_A}$.
Then for every $f\in L^{p}((0,T),v;X_{0})$ and $x\in X_{v,p}$, there exists a unique strong solution $u\in \MRT$ of
\eqref{eqprobperturbation} and
\begin{equation}\label{eqaprioriA0genTpert2}
\|u\|_{\MRT}  \leq C(1+\|x\|_{X_{v,p}}+ \|f\|_{L^{p}((0,T),v;X_{0})}),
\end{equation}
where $C$ is independent of $f$ and $x$.
\end{proposition}

The proof of this proposition is a standard application of the regularity estimate \eqref{eqaprioriA0genTpert} combined with the Banach fixed point theorem. A similar result holds on infinite time intervals if one assumes $\|B(t,x)\|_{X_{0}} \leq C_B\|x\|_{X_{1}}$.

\begin{proof}
Let $\lambda>0$ be so large that $\frac{C_A L_B}{\lambda}<C_A \varepsilon:=1-\theta$ and define the following equivalent norm on $\MRT$:
\[\|u\|_{\lambda} = \lambda \|u\|_{L^p((0,T),v;X_0)} + \|u\|_{\MRT}.\]

We will prove that for all $g\in L^p((0,T),v;X_0)$ and $x\in X_{v,p}$ there exists a unique strong solution $w\in \MRT$ of
\begin{equation}\label{eqprobperturbationlambda}
\begin{aligned}
w'(t)+(A(t)+\lambda)w(t)+\tilde{B}(t,w(t))& =g(t), \ \ \  w(0) =x.
\end{aligned}
\end{equation}
and that $w$ satisfies the estimate \eqref{eqaprioriA0genTpert2} with $(u,f)$ replaced by $(w,g)$. Here $\tilde{B}(t,x) = e^{-\lambda t} B(t,e^{\lambda t} x)$ and note that $\tilde{B}$ satisfies the same Lipschitz estimate \eqref{eq:Blipschitz} as $B$.
To see that the required result for \eqref{eqprobperturbation} follows from this, note that there is a
one-to-one correspondence between both problems given by $u(t) = e^{\lambda t} w(t)$ and $f = e^{\lambda t}g$. Therefore, from now it suffices to consider \eqref{eqprobperturbationlambda}.

In order to solve \eqref{eqprobperturbationlambda} we use the maximal regularity estimate \eqref{eqaprioriA0genTpert} combined with Proposition \ref{prop:initialvalue} and the special choice of $\lambda$.
For $\phi \in \MRT$ we write $w = L(\phi)$, where $w\in \MRT$ is the unique strong solution of
\begin{equation}\label{eqprobperturbationlambdafixed}
\begin{aligned}
w'(t)+(A(t)+\lambda)w(t)& =g(t) - \tilde{B}(t,\phi(t)),\ \ \ w(0) =x.
\end{aligned}
\end{equation}
Then for $\phi_1, \phi_2\in \MRT$, by \eqref{eqaprioriA0genTpert} one has
\begin{align*}
\|L(\phi_1) - L(\phi_2)\|_{\lambda} &\leq C_A \|\tilde{B}(\cdot, \phi_1) - \tilde{B}(\cdot, \phi_2)\|_{L^p((0,T),v;X_0)}
\\ & \leq C_A \varepsilon\|\phi_1-\phi_2\|_{L^p((0,T),v;X_{1})}+C_A L_B \|\phi_1-\phi_2\|_{L^p((0,T),v;X_{0})}
\\ & \leq (1-\theta)\|\phi_1-\phi_2\|_{\lambda}.
\end{align*}
Hence $L$ is a contraction on $\MRT$ with respect to the norm $\|\cdot\|_{\lambda}$. Therefore, by the Banach fixed point theorem there is a unique $w\in \MRT$ such that $L(w) = w$. It is clear that $w$ is the required strong solution of \eqref{eqprobperturbationlambda}. To prove the required estimate note that by \eqref{eqaprioriA0genTpert} and Proposition \ref{prop:initialvalue} one has
\begin{align*}
\|w\|_{\lambda} & = \|L(w)\|_{\lambda} \leq \|L(w) - L(0)\|_{\lambda} + \|L(0)\|_{\lambda}
\\ & \leq (1-\theta) \|w\|_{\lambda} + C_A (\|g\|_{L^p((0,T),v;X_0)} + C_B) + C \|x\|_{X_{v,p}}.
\end{align*}
Subtracting $(1-\theta) \|w\|_{\lambda}$ on both sides, and rewriting the estimate in terms of $f$ and $u$ gives the required result.
\end{proof}

With a similar method as in Proposition \ref{prop:perturb} one obtains the following perturbation result which will be used in the next Section \ref{sec:Quasi}.
\begin{proposition}\label{prop:maxregsmallB}
Assume $T<\infty$. Assume \ref{as:operatorA} holds and $A(\cdot)$ has maximal $L^p_v$-regularity on $(0,T)$ and the estimate \eqref{eq:MRest} holds with constant $C_A$. Let $\varepsilon<C_A$. If $B:[0,T]\to\calL(X_1, X_0)$ satisfies $\|B(t) x\|_{X_0}\leq \varepsilon \|x\|_{X_1}$ for all $x\in X_1$ and $t\in [0,T]$, then $A+B$ has maximal $L^p_v$-regularity with constant $\frac{C_A}{1-C_A \varepsilon}$.
\end{proposition}
\begin{proof}
One can argue as in the proof of Proposition \ref{prop:perturb} with $\lambda = 0$, $g = f$, $\tilde{B} = B$ and $1-\theta = C_A \varepsilon$. Moreover, if $w= L(w)$, then
\begin{align*}
\|w\|_{\MRT} & = \|L(w) - L(0)\|_{\MRT} + \|L(0)\|_{\MRT} \\ & \leq (1-\theta)\|w\|_{\MRT} + C_A \|f\|_{L^p((0,T),v;X_0)},
\end{align*}
and the required estimate result follows.
\end{proof}

Consider the sequence of problems:
\begin{equation}\label{eq:approx}
\begin{cases}
u'_n(t)+A_n(t)u(t)=f_n(t),\ t\in(a,b)\\
u(a)=x_n.
\end{cases}
\end{equation}
Here we omit the initial condition if $a=-\infty$.

Recall that $v_a = v(\cdot-a)$. The following approximation result holds.
\begin{proposition}\label{prop:approx}
Assume \ref{as:operatorA} holds for $A$ and $A_n$ for $n\geq 1$ with uniform estimates in $n$.
Assume $A$ and $A_n$ for $n\geq 1$ have maximal $L^p_v$-regularity on $(a,b)$ with uniform estimates in $n$. Let $f_n, f\in L^p((a,b),v;X_0)$ and $x_n,x\in X_{v_a,p}$ for $n\geq 1$.
Then if $u$ and $u_n$ are the solutions to \eqref{eqprobabstinvalprel} and \eqref{eq:approx}  respectively, then there is a constant $C$ only dependent on the maximal $L^p_v$ regularity constants and the constants in \ref{as:operatorA} such that
\begin{equation}\label{eq:approxun}
\begin{aligned}
\|u_n - u\|_{\MR} \leq C  \Big[\|x_n-x\|_{X_{v_a,p}} & + \|f_n - f\|_{L^p((a,b),v;X_0)} \\ &   + \|(A_n-A)u\|_{L^p((a,b),v;X_0)}\Big].
\end{aligned}
\end{equation}
In particular if $x_n\to x$ in $X_{v_a,p}$, for all $z\in X_1$, $A_n(t) z \to A(t) z$ in $X_0$ a.e.\ and $f_n\to f$ in $L^p((a,b),v;X_0)$, then $u_n\to u$ in $\MR$.
\end{proposition}

Typically, one can take $A_n = \varphi_n*A$ where $(\varphi_n)_{n\geq 1}$ is an approximation of the identity. If $\varphi_n$ are smooth functions, then $A_n$ will also be smooth and therefore, $A_n$ will generate an evolution system with many additional properties (see \cite{Lun,Ta1}).

\begin{proof}
The last assertion follows from \eqref{eq:approxun} and the dominated convergence theorem. To prove the estimate \eqref{eq:approxun} note that $w_n = u_n - u$ satisfies the following equation
\[w_n' + A_n w_n = (f_n - f) + (A_n - A) u, \ \ \  w_n(a) = x_n - x.\]
Therefore, the \eqref{eq:approxun} follows immediately from the maximal $L^p_v$-regularity estimate.
\end{proof}

\section{An example: $m$-th order elliptic operators\label{sec:2mell}}

In this section let $p,q\in(1,\infty)$, $m\in\{1,2,...\}$ and consider the usual multi-index notation $D^{\alpha}=D_{1}^{\alpha_{1}}\cdot ...\cdot D_{d}^{\alpha_{d}}$, $\xi^{\alpha}=(\xi^{1})^{\alpha_{1}}\cdot ...\cdot (\xi^{d})^{\alpha_{d}}$ and $|\alpha|=\alpha_{1}+\cdots +\alpha_{d}$ for a multi-index $\alpha=(\alpha_{1},\cdots,\alpha_{d})\in\N^{d}_0$. Below we let $X_0 = L^q(\R^d,w)$ and $X_1 = W^{m,q}(\R^d,w)$.

Recall that $f\in W^{m,q}(\R^d,w)$ if  $f\in L^q(\R^d,w)$ and for all $|\alpha|\leq m$,
$\|D^{\alpha} f\|_{L^q(\R^d,w)}<\infty$. In this case we let
\[[f]_{W^{m,q}(\R^d,w)} = \sum_{|\alpha| = m} \|D^\alpha f\|_{L^q(\R^d,w)}, \ \ \|f\|_{W^{m,q}(\R^d,w)} = \sum_{|\alpha| \leq m} \|D^\alpha f\|_{L^q(\R^d,w)}.\]

The weights in space will be used in combination with Theorem \ref{thm:weightedR} to obtain $\Rr$-boundedness of the integrals operators arising in \ref{as:RbddT}.

\medskip

Consider an $m$-th order elliptic differential operator $A$ given by
\begin{align}\label{eq:Adef}
(A(t) u)(t,x):=\sum_{|\alpha|\leq m}a_{\alpha}(t,x)D^{\alpha} u(t,x),\ \ t\in \R_+,\ x\in\R^{d},
\end{align}
where $D_{j}:=-i\frac{\partial}{\partial_{j}}$ and $a_{\alpha}:\R_+\times\R^d\to \C$.

In this section we will give conditions under which there holds maximal $L^p_v$-regularity for $A$ or equivalently we will prove optimal $L^p_v$-regularity results for the solution to the problem
\begin{equation}\label{eqparprobhigherorder}
\begin{cases}
u'(t,x)+(\lambda+A(t))u(t,x)=f(t,x),\ t\in(a,b),\ x\in\R^{d}\\
u(a,x)=u_0(x),\ \ \ x\in\R^{d}.
\end{cases}
\end{equation}
A function $u$ will be called a {\em strong $L^p_v(L^q_w)$-solution} of \eqref{eqparprobhigherorder} if $u\in \MR$ and \eqref{eqparprobhigherorder} holds almost everywhere.\\
\indent With slight abuse of notation we write $A$ for the realization of $A$ on $X_0 = L^q(\R^d,w)$ with domain $D(A) = X_1$. In this way \eqref{eqparprobhigherorder} can be modeled as a problem of the form \eqref{eq:cauchyinitial}. Also, we have seen in Section \ref{sec:maxLp} (and in particular Proposition \ref{prop:initialvalue}) that it is more general to study maximal $L^p_v$-regularity on $\R$. Therefore,we will focus on this case below.

\subsection{Preliminaries on elliptic equations}

In this section we introduce notation and present some results for elliptic equations which will be needed below.

Let
\begin{align*}
A:=\sum_{|\alpha|\leq m}a_{\alpha}D^{\alpha},
\end{align*}
with $a_{\alpha}\in\C$ constant.
The \textit{principal symbol} of $A$ is defined as
\begin{equation}
A_{\sharp}(\xi):=\sum_{|\alpha|=m}a_{\alpha} \xi^{\alpha}.
\nonumber
\end{equation}
We say that $A$ is \textit{uniformly elliptic} of angle $\theta\in (0,\pi)$ if there exists a constant $\kappa\in (0,1)$ such that
\begin{equation}
A_{\sharp}(\xi)\subset\Sigma_{\theta}\ {\rm and}\ |A_{\sharp}(\xi)|\geq\kappa,\ \ \ \xi\in\R^{d},\ |\xi|=1.
\nonumber
\end{equation}
If additionally there is a constant $K$ such that $|a_{\alpha}|\leq K$ for all $|\alpha|\leq m$, then
we write $A\in\Ell(\theta,\kappa,K)$.

\medskip

The following result is on the sectoriality of the operator in the $x$-independent case.
The proof is an application of the Mihlin multiplier theorem.

\begin{theorem}\label{teoHHH}
Let $1<q<\infty$ and $w\in A_q$.
Assume $A \in \Ell(\theta_0,\kappa,K)$ with $\theta_0\in (0,\pi)$.
Then for every $\theta>\theta_0$ there exists an $A_{q}$-consistent constant $C$ depending on the parameters $m,d,\theta_0-\theta,\kappa,K,q$ such that
\begin{equation}\label{eqteoHHH}
\|\lambda^{1-\frac{|\beta|}{m}}D^{\beta}(\lambda+A)^{-1}\|_{\calL(L^q(\R^d,w))}\leq C,\ \ \ |\beta|\leq m,\ \lambda\in\Sigma_{\pi-\theta}.
\end{equation}
In particular, there is a constant $\tilde{C}$ depending only on $\theta$ and $C$ such that $\|e^{-tA}\|\leq \tilde{C}$.
\end{theorem}
The case of $x$-dependent coefficients can be derived by standard localization arguments, but we will not need this case below (see \cite[Theorem 3.1]{HHH} and \cite[Section 6]{krylov}).

\begin{proof}
For \eqref{eqteoHHH} we need to check that for every $\lambda\in \Sigma_{\pi-\theta}$, and $|\beta|\leq m$, the symbol $\mathcal{M}:\R^d\to \C$ given by
\[\mathcal{M}(\xi) = \lambda^{1-\frac{|\beta|}{m}} \xi^{\beta} (\lambda+A_{\sharp}(\xi))^{-1}\]
satisfies the following: for every multiindex $\alpha\in \N^d_0$,  there is a constant $C_{\alpha}$ which only depends on
$d, \alpha, \theta-\theta_0, K, \kappa$ such that
\begin{equation}\label{eq:Mihlincond}
|\xi^{\alpha}D^{\alpha} \mathcal{M}(\xi)| \leq C_{\alpha}, \ \ \ \xi\in \R^d.
\end{equation}
Indeed, as soon as this is checked, the result is a consequence of the weighted version of Mihlin's multiplier theorem (see \cite[Theorem IV.3.9]{GarciaRubio}).

In order to check the condition for $\ell\geq 0$ let $F_{\ell}$ be the span of functions of the form $\lambda^\eta g h^{-1}$, where $\eta\in [0,1]$, $g:\R^d\to \C$ is polynomial which is homogeneous of degree $\nu\in \N_0$ and $h = (\lambda+A_{\sharp})^{\mu}$ with $\mu\in \N$ and $\ell = m(\mu-\eta) - \nu$.
It is clear that $\mathcal{M}\in F_0$. Using induction one can check that for $f\in F_{\ell}$ one has $D^{\alpha} f\in F_{\ell+|\alpha|}$.

We claim that for $f\in F_{\ell}$ the mapping $\xi\mapsto |\xi|^\ell f(\xi)$  is uniformly bounded. In order to prove this it suffices to consider $f =\lambda^\eta g h^{-1}$ with $g$ and $h$ as before, and $\ell = m(\mu-\eta) - \nu$. As $\xi\mapsto |\xi|^{-\nu} g(\xi)$ is bounded it remains to estimate
\[\lambda^\eta h(\xi)^{-1} |\xi|^{\ell+\nu} = s^{\eta} (s + A_{\sharp}(\xi^*))^{-\mu} ,\]
where $\xi^* = \xi/|\xi|$ and $s = \lambda |\xi|^{-m}$.

Write $A_{\sharp}(\xi^*) = re^{i\varphi}$ with $r = |A_{\sharp}(\xi^*)|$ and $|\varphi|<\theta_0$ and $s = \rho e^{i\psi}$ with $\rho = |s|$ and $|\psi|<\pi-\theta$. Then
\begin{align*}
|s^{\eta} (s + A_{\sharp}(\xi^*))^{-\mu}| = \rho^{\eta} |\rho e^{i\psi} + r e^{i\varphi}|^{-\mu} = \rho^{\eta} (\rho^2 + r^2 + 2\rho r \cos(\psi-\varphi))^{-\mu/2}.
\end{align*}
Since $\cos(\psi-\varphi)\geq \cos(\pi-(\theta-\theta_0)) = -\cos(\theta-\theta_0) = -(1-\varepsilon^2)$ with $\varepsilon\in (0,1)$ and $-2\rho r \geq -(\rho^2 + r^2)$ and  we find
\begin{align*}
|s^{\eta} (s + A_{\sharp}(\xi^*))^{-\mu}| \leq  \rho^{\eta} (\rho^2+r^2)^{-\mu/2} \varepsilon^{-\mu}\leq \kappa^{\mu-\eta} \varepsilon^{-\mu},
\end{align*}
where in the last step we used $r\geq \kappa$ and $\mu\geq \eta$. This proves the claim.

In order to check \eqref{eq:Mihlincond} note that $\mathcal{M}\in F_0$ and hence by the above $D^{\alpha} \mathcal{M}(\xi)\in F_{|\alpha|}$. Therefore, the bound follows from the claim about $F_{\ell}$ and the observation that the functions $g$ arising in the linear combinations of the form $\lambda^\eta g h^{-1}$ satisfy $|g(\xi)|\leq C_{K,d,\alpha}$.

The assertion for $e^{-tA}$ follows from \eqref{eq:semigroupbound} and the estimate \eqref{eqteoHHH} with $\beta = 0$.
\end{proof}

As a consequence we obtain the following:
\begin{corollary}\label{cor:domainch}
Let $\lambda_0>0$. Under the conditions of Theorem \ref{teoHHH}, the operator $A$ is closed and for every $\lambda\geq \lambda_0$,
\[c\|u\|_{W^{m,q}(\R^d,w)}\leq \|(\lambda+A)u\|_{L^q(\R^d,w)}\leq (K+\lambda)\|u\|_{W^{m,q}(\R^d,w)},\]
where $c^{-1}$ is $A_{q}$-consistent and only depends on $m,d,\theta_0-\theta,\kappa,K,q$ and $\lambda_0$.
\end{corollary}
Corollary \ref{cor:domainch} for $x$-dependent coefficients will be derived from Theorem \ref{teomaxreghigherorder} in Remark \ref{rem:domainch}.

\begin{corollary}\label{cor:WandIntp}
Let $m\geq 1$, $1<q<\infty$ and $w\in A_q$.  If $m\geq 2$, then there is an $A_q$-consistent constant $C$ depending only on $d$, $q$ and $m$ such that for all $|\beta|\leq m-1$
\begin{align*}
\|D^{\beta}f\|_{L^q(\R^d,w)} & \leq C \|f\|_{L^q(\R^d,w)} [f]_{W^{m,q}(\R^d,w)}
\\ & \leq   C'\lambda^{\frac{\beta}{m}} \|f\|_{L^q(\R^d,w)} + C' \lambda^{-\frac{m-|\beta|}{|m|}}[f]_{W^{m,q}(\R^d,w)}.
\end{align*}
\end{corollary}
\begin{proof}
Note that for $|\beta|=1$,
\[\|D^{\beta}f\|_{L^q(\R^d,w)} \leq C\lambda^{\frac12} \|f\|_{L^q(\R^d,w)}  +  \lambda^{-\frac12}[f]_{W^{2,q}(\R^d,w)} \]
follows from Theorem \ref{teoHHH} with $A = -\Delta$ and the required estimate follows by minimizing over all $\lambda>0$. The case $m>2$ can be obtained by induction (see \cite[Exercise 1.5.6]{krylov}). The final estimate follows from Young's inequality.
\end{proof}

\subsection{Main result on $\R^d$}
For $A$ of the form \eqref{eq:Adef} and $x_0\in \R^d$ and $t_0\in \R$ let us introduce the notation:
\[A(t_0,x_0) := \sum_{|\alpha| \leq m} a_{\alpha}(t_0,x_0) D^{\alpha}.\]
for the operator with constant coefficients.

\begin{enumerate}
\item[\textbf{(C)}] Let $A$ be given by \eqref{eq:Adef} and assume each $a_{\alpha}:\R\times\R^d\to \C$ is measurable.
We assume there exist $\theta_0\in [0,\pi/2)$, $\kappa$ and $K$ such that for all $t_0\in \R$ and $x_0\in \R^d$, $A(t_0,x_0)\in \Ell(\theta_0,\kappa,K)$.
Assume there exists an increasing function $\omega:(0,\infty)\to (0,\infty)$ with the property $\omega(\varepsilon)\rightarrow 0$ as $\varepsilon\downarrow 0$ and such that
    \begin{align*}
    |a_{\alpha}(t,x)-a_{\alpha}(t,y)|\leq \omega(|x-y|), \ \ |\alpha| = m, \ t\in \R, \ x,y\in\R^{d}.
\end{align*}
\end{enumerate}

As $\theta_0<\pi/2$, the above ellipticity condition implies that $m$ is even in all the results below.

The set of parameters on which all constant below will depend is given by
\begin{equation}\label{eq:parameters}
  \mathcal{P} = \{\kappa, K, \omega, [v]_{A_p}, [w]_{A_q}, p, q, d, m, \theta_0\}.
\end{equation}
Moreover, all the dependence on the weights will be in an $A_p$ and $A_q$-consistent way.

\begin{theorem}\label{teomaxreghigherorder}
Let $p,q\in (1, \infty)$. Let $v\in A_p(\R)$ and $w\in A_q(\R^d)$.
Assume condition (C) on $A$. Then there exists a $\lambda_0\in \R$ depending on the parameters in $\mathcal{P}$ such that for all $\lambda\geq \lambda_0$ the operator $\lambda+A$ has maximal $L^{p}_v$-regularity on $\R$.
Moreover, for every $\lambda\geq \lambda_0$ and for every $f\in L^{p}(\R,v;X_0)$ there exists a unique $u\in \MRinfR$
which is a strong $L^p(L^q)$ solution of
\[u'(t,x)+(\lambda+A(t))u(t,x)=f(t,x),\ a.e. \ t\in\R,\ x\in\R^{d}\]
and there is constant $C$ depending on the parameters in $\mathcal{P}$ such that
\begin{equation}\label{eqmaxregspacehigherorder2}
\lambda\|u\|_{L^{p}(\R,v;X_0))}+\|u\|_{\MRinfR}\leq C \|f\|_{L^{p}(\R, v;X_0)}.
\end{equation}
\end{theorem}
Recall that $\MRinfR = W^{1,p}(\R,v;X_0)\cap L^{p}(\R,v;X_{1})$.

Also note that the estimate \eqref{eqmaxregspacehigherorder2} also holds if one replaces $\R$ by $(-\infty, T)$ for some $T\in \R$.
The above result also implies that $\lambda+A$ has maximal $L^p_v$-regularity on $(0,T)$ for every $T<\infty$ and every $\lambda\in \R$.

The proof of the above result is a based on Theorem \ref{teostep1generalized}, standard PDE techniques and extrapolation arguments. The proof of Theorem \ref{teomaxreghigherorder} is divided in several steps of which some are standard, but we prefer to give a complete proof for convenience of the reader. In Steps 1 and 2 we assume $a_{\alpha} = 0$ for $|\alpha|<m$ and show how to include these lowers order terms later on.

\medskip

{\em Step 1:}  Consider the case where the coefficients $a_{\alpha}:\R\to \C$ are $x$-independent. Choose $\delta>0$ small enough and set $A_0 = \delta (-\Delta)^{m/2}$. Note that by Corollary \ref{cor:domainch} $D(A_0) = X_1$. We write
\[A(t) = \sum_{|\alpha| = m} a_{\alpha}(t) D^{\alpha},  \  \ \tilde{A}(t) = A(t) - A_0.\]
It is a simple exercise to see that there exist $\delta_0>0$, $\theta'\in (\theta,\tfrac{\pi}{2})$ and  $\kappa'>0$ depending on $\kappa$ and $\theta$ that for all $\delta\in (0,\delta_0]$, $\tilde{A}(t)\in \Ell(\theta',\kappa',K)$. Therefore, each $\tilde{A}(t)$ satisfies the conditions of Theorem \ref{teoHHH} with constants only depending on $\delta_0, \kappa, \theta, K$. The same holds for operators of the form $\tilde{A}_{ab} := \frac{1}{b-a}\int_{a}^b \tilde{A}(t) \, dt$, where $0\leq a<b<\infty$. Note that $\tilde{A}_{ab}$ and $\tilde{A}(t)$ are resolvent commuting and have domain $X_1$.
Therefore, by Example \ref{ex:evolution} the evolution system for $\tilde{A}$ exists and is given by
\[T(t,s) = \exp\Big(-(t-s)\tilde{A}_{st}\Big), \ \  0\leq s\leq t<\infty.\]
Moreover, for all $\lambda>0$,
\begin{equation}\label{eq:checkingweighted}
\|T(t,s)\|_{\calL(L^q(\R^d,w))} \leq C,\ \ \ \ 0\leq s\leq t,
\end{equation}
where $C$ only depends on $\delta_0, \kappa, \theta, \theta_0, K, q, [w]_{A_q}$.
Since $A_0$ is also resolvent commuting with $\tilde{A}_{ab}$ and $\tilde{A}(t)$, it follows from Lemma \ref{lem:evfam} that the evolution system generated by $A$  factorizes as
\[S(t,s) = e^{-\frac12(t-s)A_0} T(t,s) e^{-\frac12(t-s)A_0}, \ \  0\leq s\leq t<\infty.\]

We check the hypothesis \ref{as:operatorA}, \ref{as:Evol} and \ref{as:RbddT} of Theorem \ref{teostep1generalized}.
Condition \ref{as:operatorA} follows Corollary \ref{cor:domainch} with $\lambda = 1$. For condition \ref{as:Evol}, recall from Section \ref{FuncCalc} that $A_0$ has a bounded $H^\infty$-calculus of angle $<\pi/2$. Moreover, $X_0=L^q(\R^d)$ has finite cotype (see \cite[Chapter 11]{DJT}). Finally, \ref{as:RbddT} follows from Theorem \ref{thm:weightedR} and \eqref{eq:checkingweighted}. Therefore, by Theorem \ref{teostep1generalized} we find there is a constant $C$ such that \eqref{eqmaxregspacehigherorder2} holds for all $\lambda\geq 1$.

\medskip

{\em Step 2:} Next we consider the case where the coefficients of $A$ are also $x$-dependent, but still with $a_{\alpha}=0$ for $\alpha<m$.
We start with a standard freezing lemma.
\begin{lemma}[Freezing lemma]\label{freezinglemmaho}
Let $\varepsilon>0$ be such that $\omega(\varepsilon)\leq \frac{1}{2C}$, where $C$ is the constant for \eqref{eqmaxregspacehigherorder2} obtained in Step 1.
If $u\in \MRinfR$ and for some $x_0\in \R^d$ for each $t\in \R$, $u(t,\cdot)$ has support in a ball $B(x_0,\varepsilon)=\{x:|x-x_0|<\varepsilon\}$, then for all $\lambda\geq 1$, the following estimate holds:
\begin{equation}\label{eqfreezinglemmaho}
\lambda\|u\|_{ L^{p}(\R, v; X_0)}+\|u\|_{\MRinfR} \leq 2C  \|(\lambda+A)u+u'\|_{ L^{p}(\R,v;X_0)}.
\end{equation}
\end{lemma}
\begin{proof}
Let $f:=(\lambda+A)u+u'$ and observe that $u'+(A(\cdot,x_0)+\lambda)u=f+(A(\cdot,x_0)-A)u$. By \eqref{eqmaxregspacehigherorder2}, we find
\begin{align*}
\lambda\|u\|_{L^{p}(\R, v; X_0)}+\|u\|_{\MRinfR} &\leq C\|f\|_{ L^{p}(\R, v; X_0)}+C\|(A(\cdot,x_0)-A)u\|_{ L^{p}(\R, v; X_0)}.
\end{align*}
Note that by the support condition on $u$ and the continuity of $x\mapsto a_{\alpha}(\cdot,x)$,
\[\|(A(t,x_0)-A(t))u(t)\|_{X_0} \leq \omega(\varepsilon) \|u(t)\|_{X_1}.\]
Therefore, $C\|(A(\cdot,x_0)-A)u\|_{ L^{p}(\R, v; X_0)}\leq \frac12 \|u\|_{\MRinfR}$ and hence
\begin{align*}
\lambda\|u\|_{ L^{p}(\R, v; X_0)}+\|u\|_{\MRinfR} &\leq C\|f\|_{ L^{p}(\R, v; X_0)}+\frac12 \|u\|_{\MRinfR}.
\end{align*}
and the result follows from this.
\end{proof}

\medskip

{\em Step 3:}
In this step we use a localization argument in the case $p=q$ to show that there is a constant $C$ such that for all $u\in \MRinfR$,
\begin{equation}\label{eq:aprioriestlocalization}
\lambda\|u\|_{ L^{q}(\R, v; X_0)}+\|u\|_{L^q(\R,v;X_1)} \leq C  \|(\lambda+A)u+u'\|_{ L^{q}(\R,v;X_0)}.
\end{equation}

(a) Take a $\phi\in C^{\infty}(\R^d)$ with $\phi\geq 0$, $\|\phi\|_{L^{q}(\R^{d})}=1$ and support in the ball $B_{\varepsilon}=\{x:|x|<\varepsilon\}$ where $\varepsilon>0$ is as in Lemma \ref{freezinglemmaho}. Note that
\begin{equation}\label{eqpartitionunityho}
|\nabla^{m}u(t,x)|^{q}=\int_{\R^{d}}|\nabla^{m}u(t,x) \phi(x-\xi)|^{q} d\xi.
\end{equation}
By the product rule, we can write
\[\nabla^{m}[u(t,x)\phi(x-\xi)]=\nabla^{m}u(t,x)\cdot\phi(x-\xi)
+ \sum_{|\alpha|\leq m-1} c_\alpha D^{\alpha}u(t,x)D^{g(\alpha)}\phi(x-\xi),
\]
with $|g(\alpha)|\leq m$ and $c_{\alpha}\geq 0$. Therefore,
\[
|\nabla^{m}u(t,x)\cdot\phi(x-\xi)|\leq |\nabla^{m}[u(t,x)\phi(x-\xi)]\\
+\tilde{C}\sum_{|\alpha|\leq m-1}|D^{\alpha}u(t,x)|,\]
where used $\sum_{|\alpha|\leq m-1}c_{\alpha}|D^{g(\alpha)}\phi(x)| \leq \tilde{C}$.
Taking $L^q(\R,v)$-norms on both sides gives
\begin{equation}\label{eq:productruleA}
\begin{aligned}
\|\nabla^{m}u\|_{L^q(\R,v;X_0)}
&=\biggl(\int_{\R^{d}}\|\nabla^{m}u  \phi(\cdot-\xi) \|_{L^q(\R;X_0)}^q d\xi\biggr)^{1/q}\\
&\leq \biggl(\int_{\R^{d}}\|\nabla^{m}(u\phi(\cdot-\xi))\|_{L^q(\R,v;X_0)}^{q} d\xi\biggr)^{1/q} + L,
\end{aligned}
\end{equation}
where $L = \tilde{C}\sum_{|\alpha|\leq m-1}\|D^{\alpha}u\|_{L^q(\R,v;X_0)}$. For each fixed $\xi$ in the case $p=q$, Lemma \ref{freezinglemmaho} applied to $x\mapsto u(t,x)\phi(x-\xi)$ yields
\begin{equation}\label{eq:inbetweennablaA}
\begin{aligned}
\|&\nabla^{m}(u(t)\phi)\|_{L^q(\R,v;X_0)} \leq \biggl(\int_{\R^{d}}\|\nabla^{m}(u \phi(\cdot-\xi))\|_{L^q(\R,v;X_0)}^{q}d\xi\biggr)^{1/q} + L
\\ & \leq C \Big(\int_{\R^d} \| (\lambda+A)(u\phi(\cdot-\xi))+u'\phi(\cdot-\xi)\|_{L^q(\R,v;X_0)}^q \, d\xi\Big)^{\frac1q} + L,
\end{aligned}
\end{equation}
Note that for each $\xi\in \R^d$,
\begin{align*}
(\lambda+A) (u \phi(\cdot-\xi))  &= \sum_{|\alpha|=m}a_{\alpha} D^{\alpha}[u \phi(\cdot-\xi)]+\lambda u \phi(\cdot-\xi)\\
&= (\lambda+A)u \cdot \phi(\cdot-\xi) + \sum_{|\alpha|\leq m-1}c_{\alpha}a_{\alpha}D^{\alpha}u D^{g(\alpha)}\phi(\cdot-\xi).
\end{align*}
Thus we also have
\begin{align*}
\Big(\int_{\R^d} & \|(\lambda+A) (u \phi(\cdot-\xi)) + u'\phi(\cdot-\xi)\|_{L^q(\R,v;X_0)}^q \, d\xi\Big)^{\frac1q}
\\ &\leq  \Big(\int_{\R^d} \|[(\lambda+A)u  +u']\phi(\cdot-\xi)\|_{L^q(\R,v;X_0)}^q \, d\xi\Big)^{\frac1q} + K L.
\\ &=  \|(\lambda+A)u  +u'\|_{L^q(\R,v;X_0)} + K L.
\end{align*}
Combining the latter with \eqref{eq:productruleA} and \eqref{eq:inbetweennablaA} gives
\begin{align*}
\|\nabla^{m}u\|_{L^q(\R,v;X_0)} \leq C \|(\lambda+A)u  +u'\|_{L^q(\R,v;X_0)} + (K+1) L,
\end{align*}
where $K$ is as in condition (C). We may conclude that
\begin{equation}\label{eq:Weststep}
\|u\|_{L^q(\R,v;X_1)} \leq C\|(\lambda+A)u  +u'\|_{L^q(\R,v;X_0)}  + C \|u\|_{L^q(\R,v;W^{m-1,q}(\R^d,w))}.
\end{equation}

To include the lower order terms, let
\[B(t) u(t,x) = \sum_{|\alpha|\leq m-1} a_{\alpha}(t,x) D^\alpha u.\]
By \eqref{eq:Weststep} with $f = (A + B +\lambda)u+u'$ and the triangle inequality, we find
\begin{equation}\label{eq:Weststep2}
\|u\|_{L^q(\R,v;X_1)} \leq C\|f\|_{L^q(\R,v;X_0)}  + C(K+1) \|u\|_{L^q(\R,v;W^{m-1,q}(\R^d,w))}.
\end{equation}

In a similar way, one sees that for all $\lambda\geq 1$
\begin{equation}\label{eq:lambdaeststep}
\lambda \|u\|_{L^q(\R,v;X_0)} \leq C\|f\|_{L^q(\R,v;X_0)}  + C(K+1) \|u\|_{L^q(\R,v;W^{m-1,q}(\R^d,w))}.
\end{equation}
In order to obtain  \eqref{eq:aprioriestlocalization} from  \eqref{eq:Weststep2} and \eqref{eq:lambdaeststep} note that it follows from the interpolation inequality from Corollary \ref{cor:WandIntp} that
for all $\nu>0$
\begin{equation}\label{eq:uinterp}
\|u\|_{W^{m-1,q}(\R^d,w)} \leq C \nu^{m-1} \|u\|_{L^q(\R^d,w)} + C \nu^{-1} \|u\|_{W^{m,q}(\R^d,w)}.
\end{equation}
Therefore, choosing $\nu$ small enough we can combine the latter with \eqref{eq:lambdaeststep} to obtain
\begin{align*}
\lambda \|u\|_{L^q(\R,v;X_0)} + \|u\|_{L^q(\R,v;X_1)}  & \leq C\|f\|_{L^q(\R,v;X_0)}\\ & \ \ \   + \frac12 \|u\|_{L^q(\R,v;X_1)} + C_{\nu}\|u\|_{L^q(\R,v;X_0)}.
\end{align*}
Setting $\lambda_0 =\max\{2C_{\nu}, 1\}$, it follows that for all $\lambda\geq \lambda_0$,
\begin{align*}
\frac12\lambda \|u\|_{L^q(\R,v;X_0)} + \frac12\|u\|_{L^q(\R,v;X_1)}  \leq C\|f\|_{L^q(\R,v;X_0)}.
\end{align*}
This clearly implies \eqref{eq:aprioriestlocalization}.

\medskip

{\em Step 4:} To extrapolate the estimate from the previous step to $p\neq q$, let $u:\R\to X_1$ be a Schwartz function. Then by \eqref{eq:aprioriestlocalization} we have for all $v\in A_q$ there exists $A_q$-consistent constants $\lambda_0, C>0$ such that for all $\lambda\geq \lambda_0$
\[\|F_{\lambda}\|_{L^q(\R,v)} \leq C  \|G_{\lambda}\|_{ L^{q}(\R,v)},\]
where $F_{\lambda} = \|u\|_{X_1}$, $G_{\lambda} = \|(\lambda+A)u+u'\|_{X_0}$. Therefore, by the extrapolation result Theorem \ref{thm:baseweight} it follows that for all $v\in A_p$ there exist a $A_p$-consistent constants $\lambda_0'$ and $C'$ such that for all $\lambda\geq \lambda_0'$,
\[\|F_{\lambda}\|_{L^p(\R,v)} \leq C'  \|G_{\lambda}\|_{ L^{p}(\R,v)},\]
This yields
\[\|u\|_{L^p(\R,v;X_1)}\leq C' \|(\lambda+A)u+u'\|_{L^{p}(\R, v;X_0)}.\]
Similarly, one proves the estimate for $\lambda\|u\|_{L^p(\R,v;X_0)}$. As $u' = (\lambda+A)u + u' - (\lambda+A)u$, \eqref{eqmaxregspacehigherorder2} with righthand side $f = (\lambda+A)u + u'$ follows.

\medskip

{\em Step 5:} Let $A$ be as in the theorem. For $s\in [0,1]$ let $A_s = s A + (1-s) (-\Delta)^{m/2}$, where we recall that $m$ is even. Then $A_s$ satisfies condition (C) with constants $\kappa$ and $K$ replaced by $\min\{\kappa,1\}$ and $\max\{K,1\}$, respectively. Therefore,
for all $\lambda\geq \lambda_0$, \eqref{eqmaxregspacehigherorder2} holds with right-hand side $f =  (\lambda + A_s)u+u'$ with a constant $C$ which does not dependent on $s$. For $s=0$ for all $\lambda\geq \lambda_0$, for every $f\in L^p(\R,v;X_0)$, one has existence and uniqueness of a strong solution $u\in \MRinfR$ to $u'+ (\lambda + A_s) u = f$ by step 1. Therefore, the method of continuity (see \cite[Theorem 5.2]{GilTru}) yields existence and uniqueness of a strong solution for every $s\in [0,1]$. Taking $s=1$, the required result follows and this completes the proof of Theorem \ref{teomaxreghigherorder}.

\begin{remark}\label{rem:domainch}
In the proof of Theorem \ref{teomaxreghigherorder} we applied Corollary \ref{cor:domainch} only for the case of $x$-independent coefficients. It is rather simple to derive Corollary \ref{cor:domainch} with $x$-dependent coefficients from Theorem \ref{teomaxreghigherorder} (cf. \cite[Exercise 4.3.13]{krylov}).
Indeed, let $A$ be $t$-independent but such that (C) holds and let $u\in W^{m,q}(\R^d,w)$. Applying \eqref{eqmaxregspacehigherorder2} to $\tilde{u}:\R\to X_1$ given by $\tilde{u}(t) = e^{-\mu |t|} u$ with $\mu>0$ and $v=1$, and letting $\mu\downarrow 0$, we find
\[\lambda \|u\|_{L^q(\R^d,w)} + \|u\|_{W^{m,q}(\R^d,w)} \leq C \|(\lambda +A) u\|_{L^q(\R^d,w)}.\]
\end{remark}

Finally we show how to derive Theorem \ref{thm:2mintro} from Theorem \ref{teomaxreghigherorder}.
\begin{proof}[Proof of Theorem \ref{thm:2mintro}]
By Theorem \ref{teomaxreghigherorder} there is a $\lambda\in \R$ such that $\lambda+A$ has maximal $L^p$-regularity on $\R$ and hence on $(0,T)$ as well. By the observation after Definition \ref{def:MR} this implies that $A$ has maximal $L^p$-regularity on $(0,T)$ and hence we can find a unique solution
\[u\in W^{1,p}(0,T;L^q(\R^d))\cap L^p(0,T;W^{m,q}(\R^d))\]
of \eqref{eq:introCauchy} with $u_0=0$.
By Proposition \ref{prop:initialvalue} with $v\equiv 1$, we can allow nonzero initial values $u_0\in X_{v,p} = (L^q(\R^d), W^{m,q}(\R^d))_{1-\frac1p,p}$ (see Example \ref{ex:chtrace}). By \cite[Theorem 6.2.4]{BergLof} or \cite[Remark 2.4.2.4]{Tr1} this real interpolation space can be identified with $B^{s}_{q,p}(\R^d)$ with $s= m(1-1/p)$. Finally, the fact that $u\in C([0,T];B^{s}_{q,p}(\R^d))$ follows from Example \ref{ex:chtrace} as well.
\end{proof}

In the next remark we compare Theorems \ref{thm:2mintro} and \ref{teomaxreghigherorder} to part of the literature on such equations.
\begin{remark}\label{rem:discussionHHKK}
\
\begin{enumerate}
\item In the case $A$ is time-independent and $v\equiv 1$, Theorem \ref{teomaxreghigherorder}
reduces to \cite[Theorem 3.1]{HHH} in case of scalar equations.
\item In a series of papers by Kim and Krylov several maximal $L^p$-regularity results for \eqref{eq:introCauchy} have been derived under VMO conditions on the coefficients. In \cite[Theorem 4.3.8]{krylov} the case $p=q$ and $m=2$ has been considered under the same assumptions under the coefficients.
Extensions to the case $1<q\leq p<\infty$ have been given in \cite{Kim08} and \cite[Chapter 7]{krylov}.
Here only VMO conditions in the space variable are required. In the $x$-independence case results for $p,q\in (1, \infty)$ can be found in \cite{Kim10pot,Krylovheat}. For further results and references in the case $p=q$ we refer to \cite{DK11}.
\end{enumerate}
\end{remark}

In future works we will consider other examples and applications of the above methods:
\begin{remark}\
\begin{enumerate}
\item One can extend Theorem \ref{teomaxreghigherorder} to systems of equations. This is more complicated as the evolution family is not explicitly given in this situation. This will be addressed in future works.
\item With standard methods one can extend the result of Theorem \ref{teomaxreghigherorder} to half spaces and domains. This will be presented elsewhere.
\item The same method gives new information on stochastic maximal $L^p$-regularity for SPDEs with coefficients which depend on time in a measurable way. The case of continuous dependence on time was considered in \cite{NVW-SIAM}.
\end{enumerate}
\end{remark}

\section{Quasilinear evolution equations\label{sec:Quasi}}
In this section we illustrate how the results of the paper can be used to study nonlinear PDEs. We extend the result of \cite{CleLi} and \cite{Pruss02} (see \cite{KPW} for the weighted setting) to the case of time-dependent operators $A$ without continuity assumptions. Our proof slightly differs from the previous ones since we can immediately deal with the nonautonomous setting. For notational simplicity we consider the unweighted setting only.

\subsection{Abstract setting}
Let $X_0$ be a Banach space and $X_1\hookrightarrow X_0$ densely, $0<T\leq T_0<\infty$, $J=[0,T]$, $J_0=[0,T_0]$ and $p\in (1,\infty)$. Let $X_p =(X_0,X_1)_{1-\frac{1}{p},p}$ equipped with the norm from \eqref{eq:Xvp}. Consider the quasi-linear problem
\begin{equation}\label{prob:quasilinear}
\begin{cases}
u'(t)+A(t,u(t))u(t)=F(t,u(t)), & t\in J\\
u(0)=x.
\end{cases}
\end{equation}
where $x\in X_p$ and
\begin{itemize}
\item $A:J_0\times X_p \rightarrow \mathcal{L}(X_1,X_0)$ is such that for each $y\in X_1$ and $x\in X_p$, $t\rightarrow A(t,x)y$ is strongly measurable and satisfies the following continuity condition: for each $R>0$ there is a constant $C(R)>0$ such that
\begin{equation}\label{eq:assAquasi}
\|A(t,x_1)y-A(t,x_2)y\|_{X_0}\leq C(R)\|x_1-x_2\|_{X_{p}}\|y\|_{X_1},
\end{equation}
with $t\in J_0,\ x_1,x_2\in X_{p},\ \|x_1\|_{X_{p}},\|x_2\|_{X_{p}}\leq R,\ y\in X_1$.
\end{itemize}
\begin{itemize}
\item $F:J_0\times X_p\rightarrow X_0$ is such that $F(\cdot,x)$ is measurable for each $x\in X_p$, $F(t,\cdot)$ is continuous for a.a.\ $t\in J_0$ and $F(\cdot,0)\in L^{p}(J_0;X_0)$ and  $F$ satisfies the following condition on Lipschitz continuity: for each $R>0$ there is a function $\phi_R\in L^{p}(J_0)$ such that
\[
\|F(t,x_1)-F(t,x_2)\|_{X_0}\leq \phi_R(t)\|x_1-x_2\|_{X_{p}},
\]
for a.a. $t\in J_0,\ x_1,x_2\in X_{p},\ \|x_1\|_{X_{p}},\|x_2\|_{X_{p}}\leq R$.
\end{itemize}

\begin{theorem}\label{thm:quasilinear}
Assume the above conditions on $A$ and $F$. Let $x_0\in X_p$ and assume that $A(\cdot,x_0)$ has maximal $L^p$-regularity. Then there is a $T\in (0,T_0]$ and radius $\varepsilon>0$ both depending on $x_0$ such that for all $x\in B_\varepsilon = \{y\in X_p: \|y-x_0\|_{X_p}\leq \varepsilon\}$, \eqref{prob:quasilinear} admits a unique solution $u\in {\rm MR}(J):=W^{1,p}(J;X_0)\cap L^{p}(J;X_1)$. Moreover, there is a constant $C$ such that for all $x,y\in B_{\varepsilon}$ the corresponding solutions $u^x$ and $u^y$ satisfy
\[\|u^x - u^y\|_{{\rm MR}(J)}\leq C\|x-y\|_{X_p}.\]
\end{theorem}
The proof will be given in Section \ref{subsec:proofthmquasi}.

\subsection{Example of a quasilinear second order equation}

Let $T_0>0$ and $J_0 = [0,T_0]$. In this section we will give conditions under which there exists a local solution of the problem:
\begin{equation}\label{eqparprobhigherorderquasi}
u'(t,x)+\sum_{|\alpha|=2}a_{\alpha}(t,x, u(t,x), \nabla u(t,x))D^{\alpha} u(t,x)=f(t,x,u(t,x), \nabla u(t,x)),
\end{equation}
with initial value $u(0,x) = u_0(x)$, $t\in J_0$, $x\in \R^d$ and where $D_{j}:=-i\frac{\partial}{\partial_{j}}$.
The main new feature here is that the above functions $a_{\alpha}$ are only measurable in time. Note that possible lower order terms $a_{\alpha}$ can be included in $f$.  We will provide an $L^p(L^q)$-theory for \eqref{eqparprobhigherorderquasi} under the following conditions on $p$ and $q$:
\begin{enumerate}[(i)]
\item Let $X_0 = L^q(\R^d), \ X_1 = W^{2,q}(\R^d), \ X_p = B^{2(1-\frac1p)}_{q,p}(\R^d)$ where $p,q\in (1, \infty)$ satisfy
\begin{equation}\label{eq:Sobolevexp}
2\Big(1-\frac{1}{p}\Big) - \frac{d}{q}>1.
\end{equation}
\end{enumerate}
This condition is to ensure the following continuous embedding holds (see \cite[Theorem 2.8.1]{Tr1})
\begin{equation}\label{eq:Sobolev}
B^{2(1-\frac1p)}_{q,p}(\R^d)\hookrightarrow C^{1+\delta}(\R^d), \ \ \ \text{for all} \ 0<\delta<2\Big(1-\frac{1}{p}\Big) - \frac{d}{q} - 1.
\end{equation}
Also note that $B^{2(1-\frac1p)}_{q,p}(\R^d) = (X_0, X_1)_{1-\frac1p,p}$ by \cite[2.4.2(16)]{Tr1}.

On $a$ and $f$ we assume the following conditions:
\begin{enumerate}[(i)]
\setcounter{enumi}{1}
\item Assume each $a_{\alpha}:J_0\times\R^d\times \R\times \R^d\to \C$ is a measurable function such that $\sup_{t,x,y,z}|a_{\alpha}(t,x,y,z)|<\infty$ and
there is an $\theta\in (0,\pi)$ and $\kappa\in (0,1)$ such that for all $t\in J_0,\ x,z\in \R^d, \ y\in \R$,
\begin{equation}
\sum_{|\alpha|=2} a_{\alpha}(t,x,y, z) \xi^{\alpha}\subset\Sigma_{\theta}\ {\rm and}\ \Big|\sum_{|\alpha|=2}a_{\alpha}(t,x,y,z) \xi^{\alpha}\Big|\geq\kappa,\ \ \ \xi\in\R^{d},\ |\xi|=1.
\nonumber
\end{equation}
\item Assume that for every $R>0$ there exists a function $\omega_R:\R_+\to \R_+$ with $\lim_{\varepsilon\downarrow 0}\omega_R(\varepsilon) = 0$ such that for all $t\in J_0,\ x_1, x_2\in \R^d,\ |y|, |z|\leq R$,
\[|a_{\alpha}(t,x_1,y,z)- a_{\alpha}(t,x_2,y,z)|\leq \omega_R(|x_1-x_2|).\]

\item Assume that for each $|\alpha|=2$ for every $R>0$ there exists a constant $C_{\alpha}(R)$ such that for all $t\in J_0,\ x\in \R^d,\ |y_1|, |y_2|\leq R$, and $|z_1|, |z_2|\leq R$,
\begin{equation}\label{eq:locaLip}
|a_{\alpha}(t,x, y_1, z_1)-a_{\alpha}(t,x, y_2, z_2)|\leq C_{\alpha}(R)(|y_1-y_2| + |z_1-z_2|),
\end{equation}

\item Assume $f:J_0\times\R^d\times\R\times\R^d\to \C$ is a measurable function such that
\[\int_{J_0} \Big(\int_{\R^d} |f(t,x, 0, 0)|^q \, dx\Big)^{\frac{p}{q}} \, dt <\infty.\]
For every $R>0$ there exists a function $\phi(R)\in L^p(J_0)$ such that for all $t\in J_0$, $x\in \R^d$, $|y_1|, |y_2|\leq R$ and $|z_1|, |z_2|\leq R$,
\[|f(t,x, y_1, z_1)-f(t,x, y_2, z_2)|_{X_0}\leq \phi(R)(t)(|y_1-y_2| + |z_1-z_2|).\]
\end{enumerate}

Let $\text{MR}^p(J) = W^{1,p}(J;L^q(\R^d))\cap L^p(J;W^{2,q}(\R^d))$ and note that by \eqref{eq:Sobolev}
\[\text{MR}^p(J) \hookrightarrow C(J;X_p)\hookrightarrow C(J;C^{1+\delta}(\R^d)).\]
In order to apply Theorem \ref{thm:quasilinear} to obtain local well-posedness define $A:J_0\times X_p\to \calL(X_1, X_0)$ and $F:J_0\times X_p\to X_0$ by
\begin{align*}
(A(t,v) u)(x)& =\sum_{|\alpha|\leq 2}a_{\alpha}(t,x, v(x), \nabla v(x))D^{\alpha} u(x),
\\ F(t,u)(x)& =f(t,x, u(x), \nabla u(x)).
\end{align*}
Then $A$ and $F$ satisfy the conditions of Theorem \ref{thm:quasilinear}. Indeed, applying \eqref{eq:locaLip}  we find that for $R>0$ and $\|v_1\|_{X_p}, \|v_2\|_{X_p}\leq R$ and $u\in X_1$,
\begin{align*}
\|A(t,v_1)u - A(t,v_2)u\|_{X_0} & \leq K(R) \Big(\|v_1 - v_2\|_{X_0} + \|\nabla v_1 - \nabla v_2\|_{X_0} \Big)  \|u\|_{X_1}
\\ & \leq K(R) C\|v_1-v_2\|_{X_p} \|u\|_{X_1}.
\end{align*}
Here we have used that for $k\in \{0,1\}$ and $v\in X_p$, $\|D^{k} v\|_{\infty}\leq C \|v\|_{X_p}$ by \eqref{eq:Sobolev}.

Next we check that for every $g\in X_p$, $A(\cdot,g)$ has maximal $L^p$-regularity.  In order to do so we check that $A(t,g)$ satisfies the conditions of Theorem \ref{teomaxreghigherorder}. Indeed, let $R = \|g\|_{C^1(\R^d)}<\infty$. By \eqref{eq:Sobolev}, $g\in C^{1+\delta}(\R^d)$ and therefore,
\begin{align*}
 |a_{\alpha}& (t,x_1,g(x_1),\nabla g(x_1)) -  a_{\alpha}(t,x_2,g(x_2),\nabla g(x_2))| \\ & \leq \ \ \ |a_{\alpha}(t,x_1,g(x_1),\nabla g(x_1)) -  a_{\alpha}(t,x_2,g(x_1),\nabla g(x_1))| \\ & \ \ +  \  |a_{\alpha}(t,x_2,g(x_1),\nabla g(x_1)) -  a_{\alpha}(t,x_2,g(x_2),\nabla g(x_2))|
\\ & \leq \omega_{R}(|x_1-x_2|) + |g(x_1) - g(x_2)| + |\nabla g(x_1) - \nabla g(x_2)|
\\ & \leq \omega_{R}(|x_1-x_2|) + \|g\|_{C^{1+\delta}(\R^d)} (|x_1-x_2|  + |x_1-x_2|^{\delta}).
\end{align*}
Thus $A(t,g)$ satisfies the required continuity condition in the space variable. Hence Theorem \ref{teomaxreghigherorder} yields that $A(\cdot, u_0)$ has maximal $L^p$-regularity. The conditions on $F$ can be checked in a similar way and we obtain the following result as a consequence of Theorem \ref{thm:quasilinear}.

\begin{theorem}\label{thm:quasilinearappl}
Assume the above conditions on $p,q\in (1, \infty)$ and $a_{\alpha}$ and $f$. Let $g\in X_p := B^{2(1-\frac1p)}_{q,p}(\R^d)$ be arbitrary.
Then there is a $T\in (0,T_0]$ and radius $\varepsilon>0$ both depending on $g$ such that for all $u_0\in B_\varepsilon = \{v\in X_p: \|v-g\|_{X_p}\leq \varepsilon\}$, \eqref{eqparprobhigherorderquasi} admits a unique solution
\[u\in W^{1,p}(J;L^q(\R^d))\cap L^{p}(J;W^{2,q}(\R^d))\cap C(J;X_p).\]
Moreover, there is a constant $C$ such that for all $u_0,v_0\in B_{\varepsilon}$ the corresponding solutions $u$ and $v$ satisfy
\[\|u - v\|_{W^{1,p}(J;L^q(\R^d))} + \|u - v\|_{L^{p}(J;W^{2,q}(\R^d))} + \|u - v\|_{C(J;X_p)} \leq C\|x-y\|_{X_p}.\]
\end{theorem}

\begin{remark}\
\begin{enumerate}
\item If $a_{\alpha}$ only depends on $u$ and not on its derivatives, then one can replace \eqref{eq:Sobolevexp} by the condition $2\Big(1-\frac{1}{p}\Big) - \frac{d}{q}>0$.
\item Theorem \ref{thm:quasilinearappl} can be extended to higher order equations. Then $a_{\alpha}$ is allowed to depend on the $(m-1)$-th derivatives of $u$. Moreover, by \cite{GV2} one can also consider higher order systems.
\end{enumerate}
\end{remark}

\subsection{Proof of Theorem \ref{thm:quasilinear}\label{subsec:proofthmquasi}}
From the trace estimate \eqref{eq:tracestalpha} and a simple reflection argument one sees see that there exists a constant $C$ independent of $T$ such that for all $u\in \text{MR}^p(J)$ with $u(0) = 0$ one has
\begin{equation}\label{eq:traceindepT}
\|u\|_{C([0,T];X_p)}\leq C_{\rm Tr} \|u\|_{\text{MR}^p(J)}.
\end{equation}
Without the assumption $u(0)=0$, one still has the above estimate but with a constant which blows up as $T\downarrow 0$ (see \eqref{eqaprioriA0genInitial} and use a translation argument).

For $u,v\in {\rm MR}(J)$ with $u(0) = v(0)\in X_p$, the following consequence of \eqref{eq:traceindepT} will be used frequently:
\begin{equation}\label{eq:conseqhandy}
\begin{aligned}
\|u\|_{C(J;X_p)}&\leq \|u-v\|_{C(J;X_p)} + \|v\|_{C(J;X_p)}
\\ & \leq C_{\rm Tr}\|u-v\|_{\text{MR}^p((0,T))} + \|v\|_{C(J;X_p)}.
\end{aligned}
\end{equation}

\begin{proof}[Proof of Theorem \ref{thm:quasilinear}]
We modify the presentation in \cite{KPW} to our setting. By the assumption and Proposition \ref{prop:initialvalue} we know that for each $x\in X_p$, there exists a unique solution $w^{x}\in {\rm MR}^p(J)$ of the problem
\[
\begin{cases}
w'(t)+A(t,x_0)w(t)=F(t,x_0), & t\in J_0\\
w(0)=x.
\end{cases}
\]
Moreover, by linearity
\[\|w^{x} - w^y\|_{{\rm MR}^p(J_0)}\leq C_0 \|x-y\|_{X_p}.\]
By \eqref{eqaprioriA0genInitial} and a translation argument we see that
\begin{equation}\label{eq:estwxwy}
\|w^x - w^y\|_{C(J_0;X_p)}\leq C_1\|w^{x} - w^y\|_{{\rm MR}^p(J_0)}\leq C_1 C_2\|x-y\|_{X_p}.
\end{equation}

{\em Step 1}. Let $C_A$ be the maximal $L^p$-regularity constant of $A(\cdot,u_0)$. We show that for a certain set of function $u\in {\rm MR}^p(J)$ maximal $L^p$-regularity holds with constant $2C_A$.
Fix $R>0$.
Since $w^{x_0}:[0,T]\to X_p$  is continuous we can find $T\in (0,T_0]$ such that
\begin{equation}\label{eq:wx0close}
\|w^{x_0}(t) - x_0\|_{X_p}\leq \frac{1}{4 C(R) C_A}, \ \ t\in [0,T].
\end{equation}
Let
\[r_0:=\frac{1}{4 C(R) C_A (C_{\rm Tr} +  C_{\rm Tr}C_2+  C_1 C_2)}\]
and write
\[\mathbb{B}_{r_0} = \{v\in {\rm MR}^p(J): \|v(0) - x_0\|_{X_p}\leq r_0 \ \text{and} \  \|v-w^{x_0}\|_{{\rm MR}^p(J)}\leq r_0\}.\]
From the assumptions we see that for all $v\in \mathbb{B}_{r_0}$ and $t\in [0,T]$, writing $x = v(0)$,
\begin{equation}\label{eq:vminx0}
\begin{aligned}
&\|v(t)-x_0\|_{X_p} \\ & \leq  \|v(t) - w^{x}(t)\|_{X_p} + \|w^{x}(t) - w^{x_0}(t)\|_{X_p} + \|w^{x_0}(t)-x_0\|_{X_p}
\\ & \leq C_{\rm Tr} \|v - w^x\|_{{\rm MR}^p(J)} + C_1 C_2 \|x-x_0\|_{X_p} + \frac{1}{4 C(R) C_A}
\\ & \leq C_{\rm Tr} r_0 +  C_{\rm Tr}\|w^{x_0}-w^{x}\|_{{\rm MR}^p(J)} +   C_1 C_2 \|x-x_0\|_{X_p} + \frac{1}{4 C(R) C_A}
\\ & \leq C_{\rm Tr} r_0 +   C_{\rm Tr} C_2 r_0   + C_1 C_2 r_0 + \frac{1}{4 C(R) C_A} \leq \frac{1}{2 C(R) C_A},
\end{aligned}
\end{equation}
where we used \eqref{eq:traceindepT} and \eqref{eq:estwxwy}.
Therefore by \eqref{eq:vminx0} and the assumption
\begin{align*}
\|&A(t,v(t)) - A(t,x_0)\|_{\calL(X_1, X_0)} \leq C(R) \|v(t)-x_0\|_{X_p} \leq \frac{1}{2 C_A}.
\end{align*}
Now Proposition \ref{prop:maxregsmallB} yields that $A(\cdot,v(\cdot))$ has maximal $L^p$-regularity with constant $2C_A$ for each $v\in \mathbb{B}_{r_0}$.

{\em Step 2}. Let $R = 1 + C_{\rm Tr} + C_{\rm Tr}C_2 +C_1 C_2 + \|w^{x_0}\|_{C(J_0;X_p)}$. Fix $0<r\leq \min\{1,r_0\}$ and $T$ as in Step $1$.
Note that by \eqref{eq:conseqhandy} and \eqref{eq:estwxwy} for $v\in \mathbb{B}_r$ and $x = v(0)$,
\begin{align*}
& \|v\|_{C(J;X_p)} \leq C_{\rm Tr}\|v - w^x\|_{{\rm MR}^p(J)} + \|w^x-w^{x_0}\|_{C(J;X_p)} + \|w^{x_0}\|_{C(J;X_p)}
\\ & \leq C_{\rm Tr}\|v - w^{x_0}\|_{{\rm MR}^p(J)} + C_{\rm Tr}\|w^{x_0} - w^x\|_{{\rm MR}^p(J)} + C_1 C_2 \|x-x_0\|_{X_p} + \|w^{x_0}\|_{C(J;X_p)}
\\ & \leq C_{\rm Tr} r + (C_{\rm Tr} C_2 + C_1 C_2) \|x-x_0\|_{X_p} + \|w^{x_0}\|_{C(J;X_p)}
\\ & \leq C_{\rm Tr} r + (C_{\rm Tr} C_2 + C_1 C_2) r + \|w^{x_0}\|_{C(J;X_p)}\leq R,
\end{align*}
where we used $r\leq 1$. Similarly, for $x\in B_r$, $\|x\|_{X_p}\leq r+\|x_0\|_{X_p}\leq R$.

For $x\in B_r$, let $\mathbb{B}_{r,x} \subseteq \mathbb{B}_r$ be defined by
\[\mathbb{B}_{r,x} = \{u\in {\rm MR}^p(J): u(0) = x \ \text{and} \  \|u-w^{x_0}\|_{{\rm MR}^p(J)}\leq r\}.\]
Before we introduce a fixed point operator argument on $\mathbb{B}_{r,x}$, let
\begin{align*}
f(v_1, v_2) &= F(t,v_1(t))-F(t,v_2(t)), \\
a(v_1, v_2, v_3)(t) &= (A(t,v_2(t)) - A(t,v_1(t))) v_3(t).
\end{align*}
for $v_{j} \in \mathbb{B}_{r,x_j}$ with $x_j\in B_r$ for $j\in \{1, 2\}$ and $v_3\in {\rm MR}^p(J)$.
Observe that by \eqref{eq:conseqhandy} and \eqref{eq:estwxwy}
\begin{align*}
\|v_{1} - v_{2}\|_{C(J;X_p)} & \leq C_{\rm Tr}\|v_{1} - v_{2} - (w^{x_1} - w^{x_2})\|_{{\rm MR}^p(J)} + \|w^{x_1} - w^{x_2}\|_{C(J;X_p)}
\\ & \leq C_{\rm Tr}\|v_{1} - v_{2}\|_{{\rm MR}^p(J)} +  (C_{\rm Tr} C_2 + C_1 C_2)\|x_1 - x_2\|_{X_p}.
\end{align*}
Let $C_J = \|\phi_R\|_{L^p(J)}$. For $f$ we find
\begin{align*}
\|f(v_1, v_2)\|_{L^p(J;X_0)} &\leq \|\phi_R (v_1-v_2)\|_{L^p(J;X_p)}
\leq C_J \|v_{1} - v_{2}\|_{C(J;X_p)}
\\ & \leq C_J C_{\rm Tr}\|v_{1} - v_{2}\|_{{\rm MR}^p(J)} +  C_J C_3\|x_1 - x_2\|_{X_p},
\end{align*}
where $C_3 = (C_{\rm Tr} C_2 + C_1 C_2)$. Similarly, applying the estimate for $v_1-v_2$ again,
\begin{align*}
\|a(v_1, v_2, v_3)\|_{L^p(J;X_0)} &\leq C(R) \big\| \|v_2- v_1\|_{X_p} \|v_3\|_{X_1} \|\big\|_{L^p(J;X_0)}
\\ & \leq C(R) \|v_{1} - v_{2}\|_{C(J;X_p)} \|v_3\|_{{\rm MR}^p(J)}
\\ & \leq C(R)\|v_3\|_{{\rm MR}^p(J)} \big[C_{\rm Tr}\|v_{1} - v_{2}\|_{{\rm MR}^p(J)} +  C_3\|x_1 - x_2\|_{X_p}\big].
\end{align*}

For $v\in \mathbb{B}_{r,x}$ and $x\in B_r$ let $L_x(v) = u\in {\rm MR}^p(J)$ denote the solution of
\[
\begin{cases}
u'(t)+A(t,v(t))u(t)=F(t,v(t)), & t\in J_0\\
u(0)=x.
\end{cases}
\]
For $v_{1}, v_{2}$ as before let $u_j:=L_{x_j}(v_j)$ for $j\in \{1, 2\}$. We find that $u := u_1 - u_2$ in ${\rm MR}^p(J)$ satisfies $u(0) = x_1-x_2$ and
\[u'(t) + A(t,v_1(t))u(t)= f(v_1, v_2)(t) + a(v_1, v_2, u_2)(t), \ \ t\in J.\]
Therefore, by Step 1, Proposition \ref{prop:initialvalue} and the previous estimates, we find
\begin{equation}\label{eq:LLipx12}
\begin{aligned}
\|&L_{x_1}(v_1) -L_{x_2} (v_2)\|_{{\rm MR}^p(J)} \\ & \leq 2C_A\Big( \|x_1-x_2\|_{X_p} + \|f(v_1, v_2)\|_{L^p(J;X_0)} + \|a(v_1, v_2, u_2)\|_{L^p(J;X_0)}\Big)
\\ & \leq K_1(\|u_2\|_{{\rm MR}^p(J)}) \|x_1-x_2\|_{X_p} + K_2(\|u_2\|_{{\rm MR}^p(J)})\|v_{1} - v_{2}\|_{{\rm MR}^p(J)},
\end{aligned}
\end{equation}
where for $s\geq 0$,
\begin{align*}
K_1(s) &= 2C_A (1+C_J C_3 +C(R) C_3 s),
\\ K_2(s) & = 2C_A\big[C_J C_{\rm Tr} +  C(R) C_{\rm Tr} s \big].
\end{align*}

Extending the definitions of $L$, $f$ and $a$ in the obvious way we can write $w^{x_0} = L_{x_0}(x_0)$. Estimating as before, one sees that for $x\in B_r$ and $v\in \mathbb{B}_{r,x}$,
\begin{equation}\label{eq:Lxvw}
\begin{aligned}
&\|L_{x}(v) -L_{x_0}(x_0)\|_{{\rm MR}^p(J)} \\ & \leq 2C_A\Big( \|x-x_0\|_{X_p} + \|f(v,x_0\|_{L^p(J;X_0)} + \|a(v,x_0,w^{x_0})\|_{L^p(J;X_0)}\Big)
\\ & \leq 2C_A\Big( \|x-x_0\|_{X_p} + \big[C_J + C(R) \|w^{x_0}\|_{{\rm MR}^p(J)}\big] \|v - x_0\|_{C(J;X_p)}\Big),
\\ & \leq 2C_A\Big( \|x-x_0\|_{X_p} + \big[C_J + C(R) \|w^{x_0}\|_{{\rm MR}^p(J)}\big]\frac{1}{2C(R) C_A}\Big),
\end{aligned}
\end{equation}
where in the last step we used \eqref{eq:vminx0}.

Choose $0<r\leq \min\{1,r_0\}$ such that
\[4C_A r C(R) C_{\rm Tr}\leq \frac14\]
Choose $T$ such that \eqref{eq:wx0close} holds,
\[2 C_A C_J  C_{\rm Tr}\leq \frac{1}{4}, \ \ \frac{C_J}{C(R)} \leq \frac{r}{4}, \ \ \text{and} \ \ \|w^{x_0}\|_{{\rm MR}^p(J)}\leq \frac{r}{4}.\]
Let $\varepsilon = \min\{\frac{r}{4C_A}, r\}$.
Then from \eqref{eq:Lxvw} we obtain that for all $x\in B_{\varepsilon}$, $L_{x}$ maps $\mathbb{B}_{r,x}$ into itself. In particular, for all $x\in B_{\varepsilon}$ and $v\in \mathbb{B}_{r,x}$,
\begin{equation}\label{eq:estLxv}
\|L_x(v)\|_{{\rm MR}^p(J)} \leq \|L_x(v) - w^{x_0}\|_{{\rm MR}^p(J)} + \|w^{x_0}\|_{{\rm MR}^p(J)} \leq  r + \frac{r}{4} \leq 2r.
\end{equation}
Moreover, for all $x_j\in B_{\varepsilon}$ and $v_j\in \mathbb{B}_{r,x_j}$ for $j\in \{1,2\}$,
\begin{equation}\label{eq:contractionL}
\|L_{x_1}(v_1) -L_{x_2}(v_2)\|_{{\rm MR}^p(J)} \leq K_1(2) \|x_1-x_2\|_{X_p} + \frac12 \|v_{1} - v_{2}\|_{{\rm MR}^p(J)},
\end{equation}
where we used \eqref{eq:LLipx12} and \eqref{eq:estLxv}.
In particular, $L_{x}$ defines a contraction on $\mathbb{B}_{r,x}$ and by the Banach contraction principle we find that there exists a unique $u\in \mathbb{B}_{r,x}$ such that $L_{x}(u) = u$. This yields the required result.

The final estimate of the theorem follows from \eqref{eq:contractionL}.
\end{proof}

\def\cprime{$'$}

\end{document}